\documentclass[final]{siamart1116}
\usepackage{geometry}
\usepackage{amsmath}
\usepackage{cases}
\usepackage{stmaryrd}
\usepackage{mathrsfs}
\usepackage{amssymb}
\usepackage{lipsum}
\usepackage{multirow}
\usepackage{amsfonts}
\usepackage{graphicx}
\usepackage{epstopdf}
\usepackage{algorithmic}
\usepackage{amsopn}
\usepackage{makecell,booktabs}
\usepackage{array}

\ifpdf
  \DeclareGraphicsExtensions{.eps,.pdf,.png,.jpg}
\else
  \DeclareGraphicsExtensions{.eps}
\fi
\numberwithin{theorem}{section}
\newsiamremark{remark}{Remark}
\newsiamremark{assumption}{Assumption}

\newtheorem{aproposition}{Proposition}

\newenvironment{aproof}{{\noindent\bf Proof:}}{\hfill$\Box$\medskip}

\ifpdf
\hypersetup{ pdftitle={Guide to Using  SIAM'S \LaTeX\ Style} }
\fi

\newcommand{\TheTitle}{Column $\ell_{2,0}$-norm regularized factorization model of low-rank matrix recovery and its computation}

\title{{\TheTitle}\thanks{{(revised)December 26, 2021}.
\funding{This work was supported by the National Natural Science Foundation of China
under projects No.11971177 and Guangdong Basic and Applied Basic Research Foundation (2020A1515010408).}}}

\author{Ting Tao\thanks{School of Mathematics, South China University of Technology, Guangzhou, China (\email{mattao@mail.scut.edu.cn}).}\and
  Yitian Qian\thanks{School of Mathematics, South China University of Technology, Guangzhou, China
  (\email{mayttqian@mail.scut.edu.cn}).}\and
  Shaohua Pan\thanks{School of Mathematics, South China University of Technology, Guangzhou, China (\email{shhpan@scut.edu.cn}).}
 }

\begin{document}

 \maketitle

 \begin{abstract}
  This paper is concerned with the column $\ell_{2,0}$-regularized factorization
  model of low-rank matrix recovery problems and its computation.
  The column $\ell_{2,0}$-norm of factor matrices is introduced to promote
  column sparsity of factors and low-rank solutions. For this nonconvex
  discontinuous optimization problem, we develop an alternating
  majorization-minimization (AMM) method with extrapolation,
  and a hybrid AMM in which a majorized alternating proximal method
  is proposed to seek an initial factor pair with less nonzero columns
  and the AMM with extrapolation is then employed to minimize of a smooth
  nonconvex loss. We provide the global convergence analysis
  for the proposed AMM methods and apply them to the matrix completion problem
  with non-uniform sampling schemes. Numerical experiments are conducted with
  synthetic and real data examples, and comparison results with
  the nuclear-norm regularized factorization model and the max-norm
  regularized convex model show that the column $\ell_{2,0}$-regularized
  factorization model has an advantage in offering solutions of lower error and
  rank within less time.
 \end{abstract}



 \section{Introduction}\label{sec1}

  Low-rank matrix recovery problems aim at recovering a true but unknown low-rank
  matrix $M^*\in\mathbb{R}^{n\times m}$ from as few observations as possible,
  and have wide applications in a host of fields such as statistics,
  control and system identification, signal and image processing, machine learning,
  quantum state tomography, and so on (see, e.g., \cite{Davenport16,Fazel02,GroLFBE10,Zhou15}
  and the reference therein). When the rank $r^*$ of $M^*$ or a tight upper bound for it, say
  integer $r\ge 1$, is available, these problems can be modeled as
  the following rank constrained optimization model
  \[
   \min_{X\in \mathbb{R}^{n\times m}}\Big\{f(X)\ \ {\rm s.t.}\ \ {\rm rank}(X)\le r\Big\}
  \]
  where $f\!:\mathbb{R}^{n\times m}\rightarrow \mathbb{R}_{+}$ is a loss function.
  However, in many scenarios, only the rough upper estimation $\min(m,n)$ for $r^*$
  is available to us. Now it is reasonable to consider the model
  \begin{equation}\label{rank-reg}
   \min_{X\in \mathbb{R}^{n\times m}}\Big\{f(X)+\lambda\,{\rm rank}(X)\Big\},
  \end{equation}
  which leads to a desirable low-rank solution by tuning the regularization
  parameter $\lambda>0$. Unless otherwise stated, we assume that $f$ is smooth
  and its gradient $\nabla\!f$ is Lipschitz with modulus $L_{\!f}$.

  Due to the combinatorial property of the rank function, the problem \eqref{rank-reg}
  is NP-hard and it is impossible to seek a global optimal solution
  via an algorithm with a polynomial-time complexity. A common way to deal with it
  is to achieve a desirable solution by solving its convex relaxation problem.
  For the rank regularized problem \eqref{rank-reg},
  the popular nuclear norm relaxation method (see, e.g., \cite{Candes09,Candes11,Fazel02,Recht10})
  yields a desirable solution by solving a single convex minimization problem
  \begin{align}\label{Nuclear-norm}
  \min_{X\in \mathbb{R}^{n\times m}}\Big\{f(X)+\lambda\|X\|_*\Big\}.
  \end{align}
  In the past decade, this method has made great progress in theory
  (see, e.g., \cite{Candes09,Candes11,Negahban11,Negahban12,Recht10}).
  In spite of the satisfying theoretical results, to improve its computational efficiency
  remains a challenge. In fact, almost all convex relaxation algorithms
  for \eqref{rank-reg} require an economic SVD of an $n\times m$ matrix in each iteration,
  which poses the major computational bottleneck and restricts their scalability to large-scale problems.
  Inspired by this, recent years have witnessed the renewed interest in the
  Burer-Monteiro factorization model \cite{Burer03} of low-rank optimization problems.
  By replacing $X$ with $UV^\mathbb{T}$ where $(U,V)\in\!\mathbb{R}^{n\times r}\times\mathbb{R}^{m\times r}$
  for some $r\in(r^*,\min(n,m))$, the factorized form of \eqref{Nuclear-norm} is
  \begin{equation}\label{MC-Fnorm}
  \min_{U\in \mathbb{R}^{n\times r},V\in \mathbb{R}^{m\times r}}
  F_{\lambda}(U,V):=f(UV^\mathbb{T}\!)+\frac{\lambda}{2}\big(\|U\|_F^2+\|V\|_F^2\big).
  \end{equation}

  Although the factorization form tremendously reduces the number of
  optimization variables since $r$ is usually smaller than $\min(n,m)$,
  the intrinsic bi-linearity makes the factorized objective function nonconvex
  and introduces additional critical points that are not global optimizers of the
  factored optimization problem. A recent research line for factorized
  models focuses on their nonconvex geometry landscape, especially the strict
  saddle property (see, e.g., \cite{{Bhojanapalli16},ChiGe17,Ge17,Lee19,LiWang19,Li18,Park17,Zhu181}).
  That is, every critical point of the nonconvex factorized models
  is shown to be either a local optimizer or a strict saddle point (i.e.,
  the critical point at which the Hessian matrix has a strictly negative eigenvalue).
  Another research line considers the (regularized) factorization models from a local view
  and aims to characterize the convergence rate of the iterates in terms of a certain measure
  or the growth behavior of objective functions around the set of global optimal solutions
  (see, e.g., \cite{Jain13,{Park16},SunLuo16,Tu16,ZhangSo18,Zheng16}).
  Most of these results are obtained for the factorized model under an implicit assumption
  that $r=r^*$. As we mentioned above, in many scenarios only a rough upper estimation
  is accessible to $r^*$. Thus, to ensure that these theoretical results fully work in practice,
  it is necessary to seek a factorized model involving
  a regularized term to reduce $r$ to $r^*$ automatically.

  The squared Frobenius-norm term in \eqref{MC-Fnorm} indeed plays
  such a role, and it can also reduce the ambiguities caused by invertible transformations.
  However, to achieve a low-rank solution by solving model \eqref{MC-Fnorm}, a suitably large
  $\lambda$ is necessary which, as will be shown by Proposition \ref{critpoint-lowbound} in Appendix A,
  inevitably leads to a worse error bound to the true matrix $M^*$. In fact,
  the numerical results in \cite{Fang18} also indicate that the nuclear norm regularized model
  has a worse performance on matrix completion in non-uniform sampling setting (see also
  Figure \ref{fig1} in Section \ref{sec5.3}). In view of the weakness of the nuclear norm
  to promote low rank, Shang et al. \cite{Shang2016} considered the factorization model involving the bi-trace
  and tri-trace quasi-norms of factor matrices. Their bi-trace and tri-trace
  quasi-norms are only the approximations of the rank function, and it is unclear whether
  their model is effective or not for matrix completion in non-uniform sampling.
  Note that for any $X\!\in\!\mathbb{R}^{n\times m}$ with ${\rm rank}(X)\!\le\! r$,
  \begin{equation}\label{rank-chara}
   {\rm rank}(X)=\min_{U\in\mathbb{R}^{n\times r},V\in\mathbb{R}^{m\times r}}
   \Big\{\frac{1}{2}\big(\|U\|_{2,0}+\|V\|_{2,0}\big)\ \ {\rm s.t.}\ \ X=UV^{\mathbb{T}}\Big\},
  \end{equation}
  where $\|U\|_{2,0}$ is the column $\ell_{2,0}$-norm (the number of nonzero columns) of $U$.
  This, along with the works on the zero-norm (see \cite{LuZhang13,Lu14}),
  inspires us to study the column $\ell_{2,0}$-norm regularized model
  \begin{equation}\label{MS-FL20}
  \min_{U\in \mathbb{R}^{n\times r},V\in\mathbb{R}^{m\times r}}
  \Phi_{\lambda,\mu}(U,V)\!:=f(UV^\mathbb{T}\!)
   +\frac{\mu}{2}\big(\|U\|_F^2+\|V\|_F^2\big)+\lambda\big(\|U\|_{2,0}+\|V\|_{2,0}\big),
  \end{equation}
  where $\mu>0$ is a tiny constant and the term $\frac{\mu}{2}(\|U\|_F^2+\|V\|_F^2)$
  is added to ensure that \eqref{MS-FL20} has a nonempty global optimal solution set,
  and consequently a nonempty critical point set.

  Model \eqref{MS-FL20} is a little more complicated due to the nonsmooth term
  $\|U\|_{2,0}+\|V\|_{2,0}$, but by Proposition \ref{prop1-Phi} the introduction
  of this term does not induce additional critical points. Moreover, as will be shown
  by Proposition \ref{optsol-rankbound}, the critical points of $\Phi_{\lambda,\mu}$
  associated to a suitable $\lambda$ and a tiny $\mu$ will have a rank equal to $r^*$,
  provided that their objective values are not greater than that of the projection of
  the noisy observation onto the rank $r^*$-constraint set. Some of the critical
  points of $F_{\mu}$ indeed also have a rank equal to $r^*$, but unfortunately they
  can not be identified by solving model \eqref{MC-Fnorm} solely. To the best of our knowledge,
  there is no work to present such a result for the critical points of model \eqref{MC-Fnorm}.
  In particular, along with \cite[Theorem 3.1]{TaoPanBi19}, for some classes of loss functions $f$,
  when the critical point associated to such $\lambda$ is a non-strict critical point of $F_{\mu}$,
  the solution corresponding to it will have a desirable error bound to the true $M^*$.
  In addition, by combining Proposition \ref{prop2-Phi} with \cite[Theorem 3.1]{TaoPanBi19},
  we conclude that if $f$ satisfies the assumption of \cite[Theorem 3.1]{TaoPanBi19},
  the solution corresponding to a local minimizer of rank $r^*$ of model \eqref{MS-FL20}
  has a better error bound to the true $M^*$ than the solution corresponding to a local minimizer
  of rank $r^*$ of model \eqref{MC-Fnorm} does (see also Remark \ref{remark2.1}).
  These results demonstrate the superiority of model \eqref{MS-FL20}.
 \begin{figure}[h]
  \setlength{\abovecaptionskip}{0.2pt}
  \centering
 \includegraphics[height=8.0cm,width=4.5in]{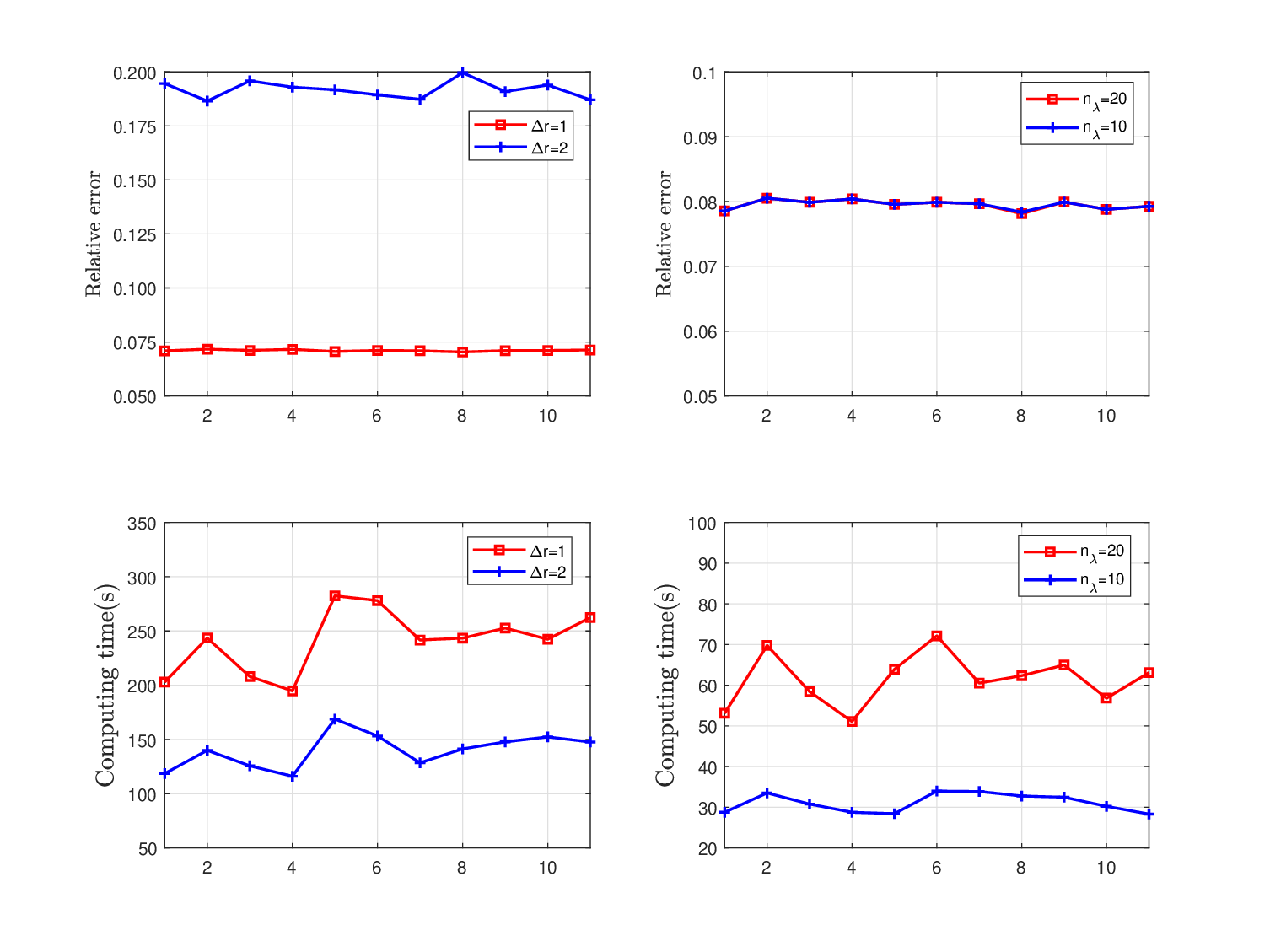}
  \caption{Solving problem \eqref{MS-FL20} by tuning $\lambda$ and minimizing $F_{\mu}$ with $\mu=10^{-8}$ by tuning $r$ (The example
  is generated randomly in the same way as in Section \ref{sec5.3} with ${\rm SR}=0.1, r^*=20$ and $n=m=2000$.)}
  \label{fig1-turning}
 \end{figure}

  Since an upper bound $r$ for $r^*$ is incorporated into model \eqref{MS-FL20},
  it is natural to ask why do not we consider model \eqref{MC-Fnorm} directly with $r$ treated as
  a tuning parameter. At first glance, tuning the integer $r\in[1,\min(m,n)]$ will be much easier
  than tuning the real number $\lambda>0$. Moreover, the popular solver ``LMaFit''
  developed by Wen et al. \cite{WenYinZhang12} for low-rank matrix recovery precisely
  uses $r$ as the tuning parameter. After running LMaFit for synthetic examples,
  we find that under the uniform sampling it outputs a solution with a rank equal to $r^*$
  whenever the initial upper bound $r$ is not too loose, but under the non-uniform sampling
  it outputs a solution with rank greater than $r^*$ even if the initial $r$ is less than
  twice of $r^*$. The subfigure on the right hand side of Figure \ref{fig1-turning}
  above shows that treating $r$ as a tuning parameter with change interval $\Delta r=1$
  will require more time, while treating $r$ with change interval $\Delta r=2$ yields a solution
  with a higher relative error since its rank is not equal to $r^*$. On the contrary, tuning
  $\lambda>0$ with change interval $\Delta\lambda=\frac{\overline{\lambda}-\underline{\lambda}}{n_{\lambda}-1}$
  can yield a desirable solution within less time, where $\overline{\lambda}$
  and $\underline{\lambda}$ are easily determined by the structure of the proximal operator
  of the $\ell_{2,0}$-norm (see Section \ref{sec5.2} for details).

  To compute the nonconvex and discontinuous model \eqref{MS-FL20}, we develop
  in Section \ref{sec3} an alternating majorization-minimization (AMM) method
  with extrapolation. Although our AMM method is a special case of the inertial
  proximal alternating linearized minimization (iPALM) method in \cite{Pock16},
  an inertial version of the PALM proposed by Bolte et al. \cite{Botle14},
  our global convergence analysis is more concise by removing the assumption on
  the boundedness of the generated sequence and quantifying the upper bound for
  the inertial parameter by the structure of $F_{\mu}$. In addition, our AMM method
  belongs to the framework of the block prox-linear method proposed by Xu and Yin \cite{XuYin17},
  but the convergence analysis there for acceleration is not applicable since
  the proximal operator of the column $\ell_{2,0}$-norm is not single-valued and
  it is unclear whether Condition 1 there holds or not for $\Phi_{\lambda,\mu}$.
  For the least squares loss $f$ from matrix completion problem, one may use
  the method proposed in \cite{YangPC18} to solve \eqref{MS-FL20},
  but the subsequential convergence there can not be
  obtained since the column $\ell_{2,0}$-norm is not continuous on its domain.
  Observe that the AMM method is actually a majorized alternating proximal (MAP) method
  with a variable metric proximal term. In Section \ref{sec4}, we first develop an MAP method
  that can yield stable nonzero column indices after a finite number of iterations, and then propose
  a hybrid AMM with a global convergence guarantee in which the MAP method is employed to
  seek an initial factor pair with less nonzero columns and the AMM with extrapolation is
  used to minimize $F_{\mu}$. The term ``global convergence'' in this work
  means the convergence of the whole sequence generated by an algorithm from any starting point.

  Finally, we apply the developed AMM methods to the matrix completion problem with
  non-uniform sampling schemes, and conduct numerical experiments with synthetic data and
  real datasets including the Jester joke, MovieLens, and Netflix datasets.
  Numerical comparisons with the alternating least squares (ALS) method \cite{Hastie15}
  for computing model \eqref{MC-Fnorm} and the ADMM
  developed in \cite{Fang18} for the SDP reformulation of the max-norm regularized
  convex model demonstrate that the AMM and the hybrid AMM for model \eqref{MS-FL20}
  have a remarkable advantage in offering solutions of lower error and rank for simulated data,
  while for the three real datasets, the hybrid AMM is superior to other three methods
  in terms of the NMAE and rank except jester-3, and it requires a comparable running time
  as ALS does and yields a desirable result for $10000\times 10000$ Netflix data in $300$ seconds.

  \noindent
  {\bf Notation:} $\mathbb{R}^{n\times m}$ represents the vector space of all $n\times m$ real matrices,
  equipped with the trace inner product $\langle X,Y\rangle={\rm trace}(X^{\mathbb{T}}Y)$
  and its induced Frobenius norm $\|\cdot\|_F$, and we stipulate $n\le m$.
  The notation $\mathbb{O}^{n\times k}$ denotes the set of matrices with orthonormal columns,
  and $\mathbb{O}^{n}$ signifies $\mathbb{O}^{n\times n}$. For a matrix $X\in\mathbb{R}^{n\times m}$,
  we denote by $\sigma(X)\in\mathbb{R}^n$ the singular value vector of $X$ arranged in a nonincreasing order,
  and by $\Sigma_{k}(X):={\rm Diag}(\sigma_1(X),\ldots,\sigma_k(X))$ the diagonal matrix consisting of
  the first $k\in[1,n]$ largest singular values. The notation $\|X\|$ and $\|X\|_*$ respectively denote
  the spectral norm and the nuclear norm of a matrix $X$, $X_i$ means the $i$th column of $X$,
  and $J_{\!X}$ and $\overline{J}_{\!X}$ denote the index set of nonzero and zero columns of $X$,
  respectively. For a self-adjoint positive semidefinite (PSD) linear operator
  $\mathcal{Q}\!:\mathbb{R}^{n\times m}\to\mathbb{R}^{n\times m}$,
  $\|\cdot\|_{\mathcal{Q}}=\sqrt{\langle \cdot,\mathcal{Q}\cdot\rangle}$ means its induced norm.
  Given a point $(\overline{U},\overline{V})\in\mathbb{R}^{n\times r}\times\mathbb{R}^{m\times r}$
  and a constant $\delta>0$, write $\mathbb{B}_{\delta}(\overline{U},\overline{V}):=\{(U,V)\,|\,\|(U,V)-(\overline{U},\overline{V})\|_F\le\delta\}$.
  For a vector $x$, $x^{\downarrow}$ means the vector consisting of the entries of $x$
  arranged in a nonincreasing order. Write $[k]:=\{1,2,\ldots,k\}$. In the sequel, we write $F(U,V):=f(UV^\mathbb{T}\!)$
  for $(U,V)\in\mathbb{R}^{n\times r}\times\mathbb{R}^{m\times r}$,
  and let $\nabla_{\!1}F(U',V')$ and $\nabla_{\!2}F(U',V')$ denote the partial gradient
  of $F$ at $(U',V')$ w.r.t. variable $U$ and $V$, respectively.
  Similarly, let $\partial_1\Phi_{\lambda,\mu}(U',V')$ and $\partial_2\Phi_{\lambda,\mu}(U',V')$
  denote the partial subdifferential of $\Phi_{\lambda,\mu}$ at $(U',V')$ w.r.t.
  variable $U$ and $V$, respectively.
 \section{Preliminaries}\label{Sec2}

  We first recall the notions of generalized subdifferentials and subderivative
  for an extended real-valued function $h\!:\mathbb{R}^p\to[-\infty,+\infty]$
  at a point with finite value.
  \begin{definition}\label{gsubdiff}
  (see \cite[Definition 8.3]{RW98})
  Consider a function $h\!:\mathbb{R}^p\to[-\infty,+\infty]$ and a point $x$
  with $h(x)$ finite. The regular subdifferential of $h$ at $x$ is defined as
  \[
    \widehat{\partial}h(x):=\bigg\{v\in\mathbb{R}^p\ \big|\
    \liminf_{x\ne x'\to x}\frac{h(x')-h(x)-\langle v,x'-x\rangle}{\|x'-x\|}\ge 0\bigg\};
  \]
  the (basic) subdifferential (also known as the limiting subdifferential) of $h$ at $x$ is
  defined as
  \[
    \partial h(x):=\Big\{v\in\mathbb{R}^p\ |\  \exists\,x^k\xrightarrow[h]{}x\ {\rm and}\
    v^k\in\widehat{\partial}h(x^k)\ {\rm with}\ v^k\to v\Big\};
  \]
  and the horizon subdifferential of $h$ at $x$, denoted by $\partial^{\infty}h(x)$, is defined as
  \[
    \partial^{\infty}h(x):=\Big\{v\in\mathbb{R}^p\ |\  \exists\,x^k\xrightarrow[h]{}x\ {\rm and}\
    v^k\in\widehat{\partial}h(x^k)\ {\rm with}\ \lambda^kv^k\to v\ {\rm for\ some}\ \lambda^k\downarrow 0\Big\},
  \]
  where the above notation $x^k\xrightarrow[h]{}x$ means $x^k\to x$ with $h(x^k)\to h(x)$.
 \end{definition}

 Let $\{(x^k,v^k)\}$ be a sequence converging to $(\overline{x},\overline{v})$
 from the graph of the mapping $\partial h$. Clearly, if $h(x^k)\to h(\overline{x})$
 as $k\to\infty$, then $(\overline{x},\overline{v})\in{\rm gph}\partial h$.
 In the sequel, a point $x\in\mathbb{R}^p$ with $0\in\partial h(x)$ is called
 a (limiting) critical point of $h$, and the set of critical points of $h$ is
 denoted by ${\rm crit}\,h$. By \cite[Theorem 10.1]{RW98}, a necessary condition
 for $\overline{x}\in\mathbb{R}^p$ to be a local minimizer of $h$ is
 $0\in\widehat{\partial}h(\overline{x})$.
 \begin{definition}\label{subderive}
 (see \cite[Definition 8.1]{RW98})
  Consider a function $h\!:\mathbb{R}^p\to[-\infty,+\infty]$ and a point $x$
  with $h(x)$ finite. The subderivative function
  $dh(x)\!:\mathbb{R}^p\to[-\infty,+\infty]$ is defined as
  \[
    dh(x)(w):=\liminf_{t\downarrow 0\atop w'\to w}\frac{h(x+tw')-h(x)}{t}
    \quad\ \forall w\in\mathbb{R}^p.
  \]
 \end{definition}
 \subsection{Subdifferentials and subderivative of column $\ell_{2,0}$-norm}\label{Sec2.1}

  The following lemma characterizes the subdifferentials and
  the subderivative function of the column $\ell_{2,0}$-norm.
 \begin{lemma}\label{gsubdiff-L20}
  Let $g(Z)\!:=\!\|Z\|_{2,0}$ for $Z\in\!\mathbb{R}^{n\times r}$.
  Fix any $(U,V)\in\!\mathbb{R}^{n\times r}\times\mathbb{R}^{m\times r}$. Then,
  \begin{itemize}
    \item [(i)] $\widehat{\partial}g(U)\!=\partial g(U)=\partial^{\infty}g(U)
                 \!=\Lambda_1\times\cdots\times \Lambda_r$ with
                 $\Lambda_i=\!\left\{\begin{array}{cl}
                   \!\{0\}^n& {\rm if}\ i\in\!J_{U},\\
                  \mathbb{R}^n &{\rm if}\ i\notin\!J_{U}.
                  \end{array}\right.$

  \item[(ii)] For any $\Gamma\in\mathbb{R}^{n\times r}$,
              \(
                dg(U)(\Gamma)=\!\left\{\begin{array}{cl}
                 0 & {\rm if}\ \overline{J}_{\!U}\cap J_{\Gamma}=\emptyset;\\
               \infty &{\rm if}\ \overline{J}_{\!U}\cap J_{\Gamma}\ne\emptyset,
               \end{array}\right.
             \)
             which means that for any $(S,W)$,
             \begin{equation*}
              d(g(U)+g(V))(S,W)
             =\left\{\begin{array}{cl}
              0 & {\rm if}\ \overline{J}_{\!U}\cap J_{\!S}=\emptyset,\overline{J}_{\!V}\cap J_{W}=\emptyset;\\
               +\infty & {\rm otherwise}.
              \end{array}\right.
             \end{equation*}
  \end{itemize}
  \end{lemma}
 \begin{proof}
  Let $\vartheta(z):={\rm sign}(\|z\|)$ for $z\in\mathbb{R}^n$.
  Fix an arbitrary $x\in\mathbb{R}^n$. Then, it holds that
  \[
   \partial^{\infty}\vartheta(x)
   =\partial\vartheta(x)
   =\widehat{\partial}\vartheta(x)
    =\left\{\begin{array}{cl}
      \!\{0\}  & {\rm if}\ x\ne0;\\
      \mathbb{R}^n &{\rm if}\ x=0
    \end{array}\right.=\big[\widehat{\partial}
   \vartheta(x)\big]^{\infty},
  \]
  where $\big[\widehat{\partial}\vartheta(x)\big]^{\infty}$
  denotes the recession cone of the closed convex set $\widehat{\partial}\vartheta(x)$.
  This by \cite[Corollary 8.11]{RW98} shows that $\vartheta$ is regular at $x$.
  In addition, for any given $w\in\mathbb{R}^n$, it is easy to calculate that
  \[
    d\vartheta(x)(w)=0\ \ {\rm when}\ \ x\ne 0
    \ \ {\rm and}\ \
    d\vartheta(0)(w)
    =\left\{\begin{array}{cl}
         0 & {\rm if}\ w=0;\\
        +\infty &{\rm if}\ w\ne 0.
    \end{array}\right.
  \]
  This means that $d\vartheta(x)(0)=0$. Together with \cite[Proposition 10.5]{RW98}
  and $g(Z)=\sum_{j=1}^r\vartheta(\|Z_j\|)$ for $Z\in\mathbb{R}^{n\times r}$,
  it is immediate to obtain part (i) and the first part of (ii). By combining
  the first part of (ii) and \cite[Proposition 10.5]{RW98}, we obtain the second
  part of (ii).
 \end{proof}

 Combining Lemma \ref{gsubdiff-L20} and \cite[Exercise 8.8]{RW98},
 we get the following characterization on $\partial\Phi_{\lambda,\mu}$.
 \begin{proposition}\label{subdiff-Phi}
  Fix any $\lambda>0$ and $\mu>0$. Consider any
  $(\overline{U},\overline{V})\!\in\mathbb{R}^{n\times r}\!\times\mathbb{R}^{m\times r}$.
  Then, it holds that
  $\widehat{\partial}\Phi_{\lambda,\mu}(\overline{U},\overline{V})
  =\partial\Phi_{\lambda,\mu}(\overline{U},\overline{V})
  =\partial_{1}\Phi_{\lambda,\mu}(\overline{U},\overline{V})
   \times\partial_{2}\Phi_{\lambda,\mu}(\overline{U},\overline{V})$ with
  \begin{subnumcases}{}\label{U-subdiff}
   \partial_{1}\Phi_{\lambda,\mu}(\overline{U},\overline{V})
   =\!\Big\{G\in\mathbb{R}^{n\times r}\,|\ G_{j}=\nabla\!f(\overline{U}\overline{V}^{\mathbb{T}})\overline{V}_{\!j}
     +\mu \overline{U}_{\!j}\ \ {\rm for}\ j\in J_{\overline{U}}\Big\},\\
  \partial_{2}\Phi_{\lambda,\mu}(\overline{U},\overline{V})
  =\!\Big\{H\in \mathbb{R}^{m\times r}\,|\ H_{j}=\big[\nabla\!f(\overline{U}\overline{V}^{\mathbb{T}})\big]^{\mathbb{T}}\overline{U}_{\!j}
  +\mu\overline{V}_{\!j}\ \ {\rm for}\ j\in J_{\overline{V}}\Big\},
  \label{V-subdiff}
 \end{subnumcases}
 which implies that $[\widehat{\partial}\Phi_{\lambda,\mu}(\overline{U},\overline{V})]^{\infty}
 =\partial^{\infty}\Phi_{\lambda,\mu}(\overline{U},\overline{V})$, and hence
 $\Phi_{\lambda,\mu}$ is a regular function.
 \end{proposition}
 \subsection{Properties of critical points to $\Phi_{\lambda,\mu}$}\label{sec2.2}

 From equations \eqref{U-subdiff}-\eqref{V-subdiff}, it is easy to check that the critical
 point set of $\Phi_{\lambda,\mu}$ defined on $\mathbb{R}^{n\times r}\times\mathbb{R}^{m\times r}$
 is strictly contained in that of $\Phi_{\lambda,\mu}$ defined on
 $\mathbb{R}^{n\times\kappa}\times\mathbb{R}^{m\times\kappa}$ for $r<\kappa$.
 The critical points of $\Phi_{\lambda,\mu}$ also have the following properties.
 \begin{proposition}\label{prop1-Phi}
  Fix any $\lambda>0$ and $\mu>0$. Then, ${\rm crit}\Phi_{\lambda,\mu}={\rm crit}F_{\mu}$,
  which means that every $(\overline{U},\overline{V})\in{\rm crit}\Phi_{\lambda,\mu}$
  satisfies $\overline{U}^\mathbb{T}\overline{U}=\overline{V}^\mathbb{T}\overline{V}$
  and $\sigma(\overline{U})=\sigma(\overline{V})$. In addition, every global
  minimizer $(\overline{U}^*,\overline{V}^*)$ of the function $\Phi_{\lambda,\mu}$ also satisfies
  ${\rm rank}(\overline{U}^*)=\|\overline{U}^*\|_{2,0}=\|\overline{V}^*\|_{2,0}={\rm rank}(\overline{V}^*)$.
 \end{proposition}
 \begin{proof}
  Pick any $(\overline{U},\overline{V})\in{\rm crit}\Phi_{\lambda,\mu}$.
  If one of $J_{\overline{U}}$ and $J_{\overline{V}}$ is empty, then they are both empty.
  Indeed, if $J_{\overline{U}}=\emptyset$ but $J_{\overline{V}}\ne\emptyset$,
  by Proposition \ref{subdiff-Phi},
  $\big[\nabla\!f(\overline{U}\overline{V}^{\mathbb{T}})\big]^{\mathbb{T}}\overline{U}_{\!j}+\mu\overline{V}_{\!j}=0$
  for $j\in J_{\overline{V}}$, which yields a contradiction $\overline{V}_{\!j}=0$ for $j\in J_{\overline{V}}$.
  Similarly, if $J_{\overline{V}}=\emptyset$, then $J_{\overline{U}}=\emptyset$ also holds.
  Thus, when either of $J_{\overline{U}}$ and $J_{\overline{V}}$ is empty,
  we have $\overline{U}=0$ and $\overline{V}=0$, which implies that
  $\nabla\!F_{\mu}(\overline{U},\overline{V})=0$.
  When $J_{\overline{U}}\ne\emptyset$ and $J_{\overline{V}}\ne\emptyset$,
  by Proposition \ref{subdiff-Phi} it holds that
  \[
    \nabla\!f(\overline{U}\overline{V}^{\mathbb{T}})\overline{V}_{\!j}+\mu\overline{U}_{\!j}=0
    \ \ {\rm for}\ j\in J_{\overline{U}}\ \ {\rm and}\ \
    \big[\nabla\!f(\overline{U}\overline{V}^{\mathbb{T}})\big]^{\mathbb{T}}\overline{U}_{\!j}+\mu\overline{V}_{\!j}=0
    \ \ {\rm for}\ j\in J_{\overline{V}}.
  \]
  The first equality implies that $\overline{V}_{\!j}\ne 0$ for $j\in J_{\overline{U}}$,
  and then $J_{\overline{U}}\subseteq J_{\overline{V}}$. The second equality
  implies that $\overline{U}_{\!j}\ne 0$ for $j\in J_{\overline{V}}$,
  and then $J_{\overline{V}}\subseteq J_{\overline{U}}$. Thus,
  $J_{\overline{U}}=J_{\overline{V}}:=J$. Consequently,
  \begin{equation*}
   \left\{\begin{array}{ll}
    \nabla\!f(\overline{U}\overline{V}^{\mathbb{T}})\overline{V}_{\!J}+\mu\overline{U}_{\!J}=0; \\
   \ [\nabla\!f(\overline{U}\overline{V}^{\mathbb{T}})]^{\mathbb{T}}\overline{U}_{\!J}+\mu\overline{V}_{\!J}=0,
   \end{array}\right.
  \end{equation*}
  which implies that $\nabla\!F_{\mu}(\overline{U},\overline{V})=0$ and
  $(\overline{U},\overline{V})\in{\rm crit}F_{\mu}$.
  Conversely, pick any $(\overline{U},\overline{V})\in{\rm crit}F_{\mu}$. Then
  $\nabla\!f(\overline{U}\overline{V}^{\mathbb{T}})\overline{V}+\mu\overline{U}=0$
  and $[\nabla\!f(\overline{U}\overline{V}^{\mathbb{T}})]^{\mathbb{T}}\overline{U}+\mu\overline{V}=0$.
  By Proposition \ref{subdiff-Phi}, $(0,0)\in \partial\Phi_{\lambda,\mu}(\overline{U},\overline{V})$,
  which means that $(\overline{U},\overline{V})\in{\rm crit}\Phi_{\lambda,\mu}$.
  From ${\rm crit}\Phi_{\lambda,\mu}={\rm crit}F_{\mu}$ and \cite[Lemma 2.2]{TaoPanBi19},
  it follows that every $(\overline{U},\overline{V})\in{\rm crit}\Phi_{\lambda,\mu}$
  satisfies $\overline{U}^\mathbb{T}\overline{U}=\overline{V}^\mathbb{T}\overline{V}$
  and $\sigma(\overline{U})=\sigma(\overline{V})$.

  Let $(\overline{U}^*,\overline{V}^*)$ be a global minimizer of $\Phi_{\lambda,\mu}$.
  Then ${\rm rank}(\overline{U}^*)={\rm rank}(\overline{V}^*)$ and $\|\overline{U}^*\|_{2,0}=\|\overline{V}^*\|_{2,0}$.
  Since $(\overline{U}^*,\overline{V}^*)$ is a global optimal solution of \eqref{MS-FL20},
  we deduce from \eqref{rank-chara} that $X^*=\overline{U}^*(\overline{V}^*)^\mathbb{T}$
  satisfies ${\rm rank}(X^*)=\frac{1}{2}(\|\overline{U}^*\|_{2,0}+\|\overline{V}^*\|_{2,0})$.
  If not, ${\rm rank}(X^*)<\frac{1}{2}(\|\overline{U}^*\|_{2,0}+\|\overline{V}^*\|_{2,0})$.
  Let $X^*$ have the SVD as $X^*=P\Sigma Q^\mathbb{T}$ for $P\in\mathbb{O}^{n}$
  and $Q\in\mathbb{O}^{m}$. Write $\widehat{U}=P_1\Sigma_1^{1/2}$ and $\widehat{V}=Q_1\Sigma_1^{1/2}$,
  where $P_1$ and $Q_1$ are the matrix consisting of the first $r$ columns of $P$ and $Q$, respectively,
  and $\Sigma_1$ is a diagonal matrix consisting of the first $r$ singular values. Then,
  by noting that $F_{\mu}(\widehat{U},\widehat{V})=F_{\mu}(\overline{U}^*,\overline{V}^*)$,
  \[
   \Phi_{\lambda,\mu}(\widehat{U},\widehat{V})=F_{\mu}(\widehat{U},\widehat{V})+\lambda(\|\widehat{U}\|_{2,0}+\|\widehat{V}\|_{2,0})\\
    =F_{\mu}(\widehat{U},\widehat{V})+2\lambda{\rm rank}(X^*)<\Phi_{\lambda,\mu}(\overline{U}^*,\overline{V}^*),
  \]
  a contradiction to the fact that $(\overline{U}^*,\overline{V}^*)$ is a global minimizer
  of $\Phi_{\lambda,\mu}$. Now combining ${\rm rank}(X^*)=\frac{1}{2}(\|\overline{U}^*\|_{2,0}+\|\overline{V}^*\|_{2,0})$
  with ${\rm rank}(X^*)\leq{\rm rank}(\overline{U}^*)\leq\|\overline{V}^*\|_{2,0}$
  yields the desired equalities.
 \end{proof}

  When $f$ has the form as in Proposition \ref{critpoint-lowbound} of Appendix A,
  for every $(\overline{U},\overline{V})\in{\rm crit}\Phi_{\lambda,\mu}$ we have
  \[
    \|\overline{U}\overline{V}^\mathbb{T}\!-M^*\|_F
    \geq\max\Big(0,\frac{\mu-\|\mathcal{A}^*\nabla h(\omega)\|}{L_{h}\|\mathcal{A}\|^2}\Big).
  \]
 Since $\mu$ is a tiny constant, this lower bound does not cause any inconsistency as
 that of Proposition \ref{critpoint-lowbound} does for model \eqref{MC-Fnorm}.
 The following proposition states that under a mild assumption on $f$,
 any critical point of model \eqref{MS-FL20} associated to a suitable $\lambda$
 and a tiny $\mu$ has a rank equal to the true $r^*$.
 \begin{proposition}\label{optsol-rankbound}
  Let $f(X)\!:=h(\mathcal{A}(X)-b)$ where $h\!:\mathbb{R}^p\to\mathbb{R}$ is a differentiable
  $\rho$-strongly convex function, $\mathcal{A}\!:\mathbb{R}^{n\times m}\to\mathbb{R}^p$
  is a linear operator and $b\in\mathbb{R}^p$ is a given vector.
  Let $M_{r^*}$ be the projection of $\mathcal{A}^*(b)$ onto the rank $r^*$-constraint set.
  Suppose the $2r^*$-restricted smallest eigenvalue $\alpha$ of $\mathcal{A}$ satisfies
  \(
   \frac{\rho\alpha\sigma_{r^*}(M_{r^*})}{2\sqrt{2}+1}\geq\|\mathcal{A}^*[\nabla h(\mathcal{A}(M_{r^*})\!-b)]\big\|_F.
  \)
  Then, for any critical point $(\overline{U},\overline{V})$ of $\Phi_{\lambda,\mu}$ associated to
  $\lambda\in[f(M_{r^*}),\frac{\rho\alpha}{16}\big(\sigma_{r^*}(M_{r^*})-
  \!\frac{1}{\rho\alpha}\|\mathcal{A}^*(\nabla h(\mathcal{A}(M_{r^*})\!-b))\|)^2]$
  and small enough $\mu>0$ such that
  $\Phi_{\lambda,\mu}(\overline{U},\overline{V})\le f(M_{r^*})+\mu\|M_{r^*}\|_*+2\lambda r^*$,
  it holds that ${\rm rank}(\overline{U}\overline{V}^{\mathbb{T}})=r^*$.
  \end{proposition}
  \begin{proof}
   Write $\overline{M}:=\!M_{r^*}-(\rho\alpha)^{-1}\mathcal{A}^*[\nabla h(\mathcal{A}(M_{r^*})\!-b)]$.
   Without loss of generality, we assume
  \[
    0<\mu <\min\Big\{\frac{\lambda}{\|M_{r^*}\|_*},\frac{\rho\alpha \sigma_{r^*}(\overline{M})}{2\sqrt{r^*}},
    \frac{\rho\alpha \sigma_{\min}(\overline{M})}{8},\frac{{\rho^2}\alpha^2\sigma^2_{r^*}(\overline{M})}
    {16\|\mathcal{A}^*[\nabla h(\mathcal{A}(M_{r^*})-b)]\|_*}\Big\},
  \]
  where $\sigma_{\min}(\overline{M})$ is the smallest nonzero singular value of $\overline{M}$.
  From $\Phi_{\lambda,\mu}(\overline{U},\overline{V})\le
  \!f(M_{r^*})\!+\mu\|M_{r^*}\|_*\!+2\lambda r^*$, we get
  $\lambda(\|\overline{U}\|_{2,0}+\|\overline{V}\|_{2,0})\le\!f(M_{r^*})+\mu\|M_{r^*}\|_*+2\lambda r^*$.
  Note that $f(M_{r^*})\le\lambda$ and $\mu<\frac{\lambda}{\|M_{r^*}\|_*}$.
  Hence, $\lambda(\|\overline{U}\|_{2,0}+\|\overline{V}\|_{2,0})< 2\lambda(r^*+1)$,
  which along with ${\rm rank}(\overline{U}\overline{V}^{\mathbb{T}})\le \|\overline{U}\|_{2,0}$
  and $\|\overline{U}\|_{2,0}=\|\overline{V}\|_{2,0}$ implies that
  ${\rm rank}(\overline{U}\overline{V}^{\mathbb{T}})\le r^*$.
  We next argue that $\overline{r}:={\rm rank}(\overline{U}\overline{V}^{\mathbb{T}})<r^*$
  can not hold. Suppose on the contradiction that $\overline{r}<r^*$. Then,
  \begin{align*}
    \Phi_{\lambda,\mu}(\overline{U},\overline{V})
    &=h(\mathcal{A}(\overline{U}\overline{V}^{\mathbb{T}})-b)
       +\frac{\mu}{2}\big(\|\overline{U}\|_F^2+\|\overline{V}\|_F^2\big)
       +\lambda\big(\|\overline{U}\|_{2,0}+\|\overline{V}\|_{2,0}\big)\\
    &\ge f(M_{r^*})+ \langle\nabla h(\mathcal{A}(M_{r^*})-b),\mathcal{A}(\overline{U}\overline{V}^{\mathbb{T}}\!-M_{r^*}) \rangle+\frac{\rho}{2}\|\mathcal{A}(\overline{U}\overline{V}^{\mathbb{T}}\!-M_{r^*})\|^2\\
    &\quad +\frac{\mu}{2}\big(\|\overline{U}\|_F^2+\|\overline{V}\|_F^2\big)
       +\lambda\big(\|\overline{U}\|_{2,0}+\|\overline{V}\|_{2,0}\big)\\
    &\ge \frac{\rho\alpha}{2}\|\overline{U}\overline{V}^{\mathbb{T}}\!-M_{r^*}\|^2
        +\langle\mathcal{A}^*[\nabla h(\mathcal{A}(M_{r^*})-b)],\overline{U}\overline{V}^{\mathbb{T}}\!-M_{r^*}\rangle\\
    &\quad +\frac{\mu}{2}\big(\|\overline{U}\|_F^2+\|\overline{V}\|_F^2\big)
       +\lambda\big(\|\overline{U}\|_{2,0}+\|\overline{V}\|_{2,0}\big)+f(M_{r^*})\\
    &\geq\frac{\rho\alpha}{2}\big\|\overline{U}\overline{V}^{\mathbb{T}}\!-M_{r^*}+(\rho\alpha)^{-1}
    \mathcal{A}^*[\nabla h(\mathcal{A}(M_{r^*})-b)]\big\|_F^2
      +\frac{\mu}{2}\big(\|\overline{U}\|_F^2+\|\overline{V}\|_F^2\big)\\
      &\quad\!+\!\lambda\big(\|\overline{U}\|_{2,0}\!+\!\|\overline{V}\|_{2,0}\big)\!
      -\!\frac{1}{2\rho\alpha}\!\big\|\mathcal{A}^*[\nabla h(\mathcal{A}(M_{r^*})-b)]\big\|_F^2\!+f(M_{r^*})\\
    &\!\geq\min_{{\rm rank}(X)\le\overline{r}}\Big\{\frac{\rho\alpha}{2}\|X\!-\!\overline{M}\|_F^2+\mu\|X\|_*\Big\}
      +\!2\lambda \overline{r}-\!\frac{1}{2\rho\alpha}\!\big\|\!\mathcal{A}^*[\nabla h(\mathcal{A}(M_{r^*})-b)]\big\|_F^2\!+f(M_{r^*})
  \end{align*}
  where the first inequality is due to the strong convexity of $h$, and
  the second one is using the fact that $\alpha$ is the $2{r^*}$-restricted
  smallest eigenvalue of $\mathcal{A}$. Note that $0<\mu<\rho\alpha\sigma_{r^*}(\overline{M})$.
  Then,
  \begin{align*}
  \min_{{\rm rank}(X)\le\overline{r}}\Big\{\frac{\rho\alpha}{2}\|X\!-\!\overline{M}\|_F^2+\mu\|X\|_*\Big\}
   &=-\frac{\overline{r}\mu^2}{2\alpha\rho}
     +\frac{\rho\alpha}{2}\sum_{i=\overline{r}+1}^n\sigma_i^2(\overline{M})-\mu\sum_{i=\overline{r}+1}^{n}\sigma_i(\overline{M})
     +\mu \|\overline{M}\|_* \\
  &\geq\!\Big(\frac{\rho\alpha}{4}\sum_{i=\overline{r}+1}^n\sigma_i^2(\overline{M})
     \!-\!\frac{\overline{r}\mu^2}{2\alpha\rho}
   \!-\!\mu\sum_{i=\overline{r}+1}^{n}\sigma_i(\overline{M})\Big)\\
  &\quad\!+\!\frac{\rho\alpha}{4}(r^*\!-\!\overline{r})\sigma_{r^*}^2(\overline{M})
       \!+\!\mu\|\overline{M}\|_*\!+\!\frac{\rho\alpha}{4}\sum_{i=r^*+1}^n\sigma_i^2(\overline{M})\\
  &\geq\!\frac{\rho\alpha}{4}(r^*\!-\!\overline{r})\sigma_{r^*}^2(\overline{M})+\mu\|\overline{M}\|_*
 \end{align*}
 where the last inequality is using $\mu<\min(\frac{\rho\alpha\sigma_{r^*}(\overline{M})}{2\sqrt{r^*}}\!,\!\frac{\rho\alpha \sigma_{\min}(\overline{M})}{8})$. From the last two inequalities,
  \begin{align*}
   \Phi_{\lambda,\mu}(\overline{U},\overline{V})
   &\ge\frac{\rho\alpha}{4}(r^*\!-\!\overline{r})\sigma_{r^*}^2(\overline{M})+\mu\|\overline{M}\|_*+\!2\lambda \overline{r}
        -\!\frac{1}{2\rho\alpha}\!\big\|\!\mathcal{A}^*[\nabla h(\mathcal{A}(M_{r^*})-b)]\!\big\|_F^2\!+f(M_{r^*})\\
   &\ge\frac{\rho\alpha}{4}(r^*\!-\!\overline{r})\sigma^2_{r^*}(\overline{M})\!+\!\mu\|M_{r^*}\|_*\!
     -\frac{\mu}{\rho\alpha}\big\|\mathcal{A}^*[\nabla h(\mathcal{A}(M_{r^*})-b)]\big\|_*\\
   &\quad+\!2\lambda \overline{r}\!+\!f(M_{r^*})-\!\frac{1}{2\rho\alpha}\big\|\mathcal{A}^*[\nabla h(\mathcal{A}(M_{r^*})-b)]\big\|_F^2
  \end{align*}
  where the second inequality is using the fact that $\|\overline{M}\|_*\ge \|M_{r^*}\|_*
  -(\rho\alpha)^{-1}\|\mathcal{A}^*[\nabla h(\mathcal{A}(M_{r^*})-b)]\|_*$.
  From the given assumption on $\alpha$, it follows that
  $ \frac{{\rho}\alpha}{16}(r^*\!-\!\overline{r})\sigma_{r^*}^2(\overline{M})
  -\frac{1}{2{\rho}\alpha}\!\|\mathcal{A}^*[\nabla h(\mathcal{A}(M_{r^*})-b)]\|_F^2\ge 0$,
  while from the range of $\mu$, we have $ \frac{{\rho}\alpha}{16}(r^*\!-\!\overline{r})\sigma_{r^*}^2(\overline{M})
  -\frac{\mu}{{\rho}\alpha}\big\|\mathcal{A}^*[\nabla h(\mathcal{A}(M_{r^*})-b)]\|_*>0$. Thus,
  \begin{align*}
    \Phi_{\lambda,\mu}(\overline{U},\overline{V})
   >\frac{\rho\alpha}{8}(r^*\!-\!\overline{r})\sigma^2_{r^*}(\overline{M})+f(M_{r^*})\!+\!\mu\|M_{r^*}\|_*\!+2\lambda\overline{r}
   \ge f(M_{r^*})\!+\!\mu\|M_{r^*}\|_*\!+2\lambda r^*
  \end{align*}
  where the second inequality is due to
  \(
   \frac{{\rho}\alpha}{8}(r^*\!-\!\overline{r})\sigma_{r^*}^2(\overline{M})\geq 2\lambda(r^*-\overline{r}),
  \)
  implied by the upper bound of $\lambda$ and the fact that
  \(
    \sigma_{r^*}(\overline{M})
    \ge\sigma_{r^*}(M_{r^*})-(\rho\alpha)^{-1}\|\mathcal{A}^*[\nabla h(\mathcal{A}(M_{r^*})-b)]\|>0.
  \)
  The last inequality contradicts the given assumption on the point $(\overline{U},\overline{V})$.
  Consequently, $\overline{r}=r^*$.
 \end{proof}

  Proposition \ref{optsol-rankbound} states that the critical points of $\Phi_{\lambda,\mu}$
  associated to a suitable $\lambda$ and a tiny $\mu$ must have rank $r^*$ if their objective values
  are not greater than $f(M_{r^*})+\mu\|M_{r^*}\|_*+2\lambda r^*$. Clearly,
  the global minimizer of $\Phi_{\lambda,\mu}$ associated to such $\lambda$ and $\mu$
  precisely belongs to this class of critical points. Although the condition involves the unknown $r^*$,
  Algorithm 1 and 3 developed in the next two sections provide an effective method for
  estimating it. By Proposition \ref{prop1-Phi}, some of the critical points of $F_{\mu}$
  also have a rank equal to $r^*$, but they can be identified only by leveraging model \eqref{MS-FL20}.
  To the best of our knowledge, there are no work to discuss
  which critical points of model \eqref{MC-Fnorm} will have a rank equal to that of the true $M^*$.
  By \cite[Theorem 3.1]{TaoPanBi19}, when $f$ satisfies the assumption there,
  if the critical point $(\overline{U},\overline{V})$ of $\Phi_{\lambda,\mu}$ associated to
  a suitable $\lambda$ (say, to guarantee that ${\rm rank}(\overline{U}\overline{V}^{\mathbb{T}})\le r^*$)
  and a tiny $\mu$ is a non-strict critical point of $F_{\mu}$, there exists a constant $\overline{c}>0$
  such that
  \begin{equation}\label{error-bound}
    \|\overline{U}\overline{V}^{\mathbb{T}}\!-M^*\|_F\leq\overline{c}\sqrt{r^*(\mu^2+\|\nabla f(M^*)\|^2)}.
  \end{equation}

  To close this section, we disclose the relation between the (strong) local minimizer
  of $\Phi_{\lambda,\mu}$ and that of $F_{\mu}$. Recall that $(\overline{U},\overline{V})$
  is a strong local minimizer of $F_{\mu}$ if $\exists\alpha>0$ and $\delta>0$ such that
 \begin{equation}\label{strong-min}
   F_{\mu}(U,V)\ge F_{\mu}(\overline{U},\overline{V})+\alpha\|(U,V)-(\overline{U},\overline{V})\|_F^2
   \quad\forall(U,V)\in\mathbb{B}_{\delta}(\overline{U},\overline{V}).
 \end{equation}
 \vspace{-0.3cm}
 \begin{proposition}\label{prop2-Phi}
  Fix any $\lambda>0$ and $\mu>0$. If $(\overline{U},\overline{V})$ is a (strong) local
  minimizer of $F_{\mu}$, then it is a (strong) local minimizer of $\Phi_{\lambda,\mu}$;
  and if $(\overline{U},\overline{V})$ is a nonzero (strong) local minimizer of $\Phi_{\lambda,\mu}$,
  then $(\overline{U}_{\!J},\overline{V}_{\!J})$ with $J=J_{\overline{U}}$
  is a (strong) local minimizer of $F_{\mu}$ defined on
  $\mathbb{R}^{n\times |J|}\times\mathbb{R}^{m\times |J|}$.
 \end{proposition}
 \begin{proof}
  Let $(\overline{U},\overline{V})$ be a strong local minimizer of $F_{\mu}$.
  There exist $\alpha>0$ and $\delta>0$ such that \eqref{strong-min}
  holds for all $(U,V)\in\mathbb{B}_{\delta}(\overline{U},\overline{V})$.
  Clearly, there exists $\delta'>0$ such that for all $(U,V)\in\mathbb{B}_{\delta'}(\overline{U},\overline{V})$,
  $\|U\|_{2,0}\geq\|\overline{U}\|_{2,0}$ and $\|V\|_{2,0}\geq\|\overline{V}\|_{2,0}$.
  Then, for any $(U,V)\in\mathbb{B}_{\varepsilon}(\overline{U},\overline{V})$
  with $\varepsilon=\min(\delta,\delta')$,
  \begin{align*}
  \Phi_{\lambda,\mu}(U,V)&=F_{\mu}(U,V)+\lambda(\|U\|_{2,0}+\|V\|_{2,0})\\
  &\ge F_{\mu}(\overline{U},\overline{V})+\lambda(\|\overline{U}\|_{2,0}+\|\overline{V}\|_{2,0})
       +\alpha\|(U,V)-(\overline{U},\overline{V})\|_F^2\\
  & =\Phi_{\lambda,\mu}(\overline{U},\overline{V})+\alpha\|(U,V)-(\overline{U},\overline{V})\|_F^2.
  \end{align*}
  This shows that $(\overline{U},\overline{V})$ is a strong local minimizer of $\Phi_{\lambda,\mu}$.
  Now let $(\overline{U},\overline{V})$ be a nonzero strong local minimizer of $\Phi_{\lambda,\mu}$.
  Clearly, $\overline{U}\ne 0$ and $\overline{V}\ne 0$.
  By Proposition \ref{prop1-Phi}, $J_{\overline{U}}=J_{\overline{V}}=J$.
  Also, there exist $\alpha>0$ and $\varepsilon>0$ such that for all
  $(U,V)\in\mathbb{B}_{\varepsilon}(\overline{U},\overline{V})$,
  \begin{align*}
   \Phi_{\lambda,\mu}(U,V)
   &\ge \Phi_{\lambda,\mu}(\overline{U},\overline{V})+\alpha\|(U,V)-(\overline{U},\overline{V})\|_F^2\\
   &= f(\overline{U}_{\!J}\overline{V}_{\!J}^\mathbb{T}\!)
     +\frac{\mu}{2}\big(\|\overline{U}_{\!J}\|_F^2+\|\overline{V}_{\!J}\|_F^2\big)+2\lambda|J|
     +\alpha\|(U,V)-(\overline{U},\overline{V})\|_F^2.
  \end{align*}
  In addition, there exists $\varepsilon'>0$ such that $\|A'\|_{2,0}=\|B'\|_{2,0}=|J|$
  for all $(A',B')\in\mathbb{B}_{\varepsilon'}(\overline{U}_{\!J},\overline{V}_{\!J})$.
  Pick any $(A,B)\in\mathbb{B}_{\widehat{\varepsilon}}(\overline{U}_{\!J},\overline{V}_{\!J})$
  with $\widehat{\varepsilon}=\min(\varepsilon',\varepsilon)$.
  Let $(U,V)\in\mathbb{R}^{n\times r}\times\mathbb{R}^{m\times r}$ with
  $U_{\!J}=A,U_{\!\overline{J}}=0$ and $V_{J}=B,V_{\!\overline{J}}=0$.
  Together with the last inequality, it follows that
  \begin{align*}
   F_{\mu}(A,B)&=f(AB^\mathbb{T}\!)+\frac{\mu}{2}\big(\|A\|_F^2+\|B\|_F^2\big)
    =f(UV^\mathbb{T}\!)+\frac{\mu}{2}\big(\|U\|_F^2+\|V\|_F^2\big)\\
   &=\Phi_{\lambda,\mu}(U,V)-2\lambda|J|\\
   &\ge f(\overline{U}_{\!J}\overline{V}_{\!J}^\mathbb{T}\!)
   +\frac{\mu}{2}\big(\|\overline{U}_{\!J}\|_F^2+\|\overline{V}_{\!J}\|_F^2\big)
    +\alpha\|(U,V)-(\overline{U},\overline{V})\|_F^2\\
   &=F_{\mu}(\overline{U}_{\!J},\overline{V}_{\!J})+\alpha\|(A,B)-(\overline{U}_{\!J},\overline{V}_{\!J})\|_F^2.
  \end{align*}
  This shows that $(\overline{U}_{\!J},\overline{V}_{\!J})$ is
  a strong local minimizer of $F_{\mu}$ defined on $\mathbb{R}^{n\times |J|}\times\mathbb{R}^{m\times |J|}$.
  The above arguments with $\alpha=0$ yield the conclusion on the local minimizer
  of $F_{\mu}$ and $\Phi_{\lambda,\mu}$.
 \end{proof}

 \begin{remark}\label{remark2.1}
  Let $(\overline{U},\overline{V})$ be a local minimizer of $\Phi_{\lambda,\mu}$ with
  ${\rm rank}(\overline{U}\overline{V}^{\mathbb{T}})=r^*$.
  From the second part of Proposition \ref{prop2-Phi} and \cite[Theorem 3.1]{TaoPanBi19},
  we deduce that $\overline{U}\overline{V}^{\mathbb{T}}$ has an error bound to $M^*$
  as in \eqref{error-bound} whenever $f$ satisfies the assumption of \cite[Theorem 3.1]{TaoPanBi19}.
  Similarly, if $(\widehat{U},\widehat{V})$ is a local minimizer of $F_{\lambda}$
  with ${\rm rank}(\widehat{U}\widehat{V}^{\mathbb{T}})\le r^*$ and $f$ satisfies the assumption
  of \cite[Theorem 3.1]{TaoPanBi19}, then
  \[
    \|\widehat{U}\widehat{V}^{\mathbb{T}}\!-M^*\|_F\leq c'\sqrt{r^*(\lambda^2+\|\nabla f(M^*)\|^2)}
    \ \ {\rm for\ some}\ c'>0.
  \]
  As discussed in Appendix A, only a suitably large $\lambda$ is enough to ensure that
  ${\rm rank}(\widehat{U}\widehat{V}^{\mathbb{T}})\leq r^*$. Thus,
  when $f$ satisfies the assumption of \cite[Theorem 3.1]{TaoPanBi19},
  the solution $\overline{U}\overline{V}^{\mathbb{T}}$ associated to a local minimizer
  $(\overline{U},\overline{V})$ of model \eqref{MS-FL20} with ${\rm rank}(\overline{U})=r^*$
  has a better error bound to the true $M^*$ than the solution $\widehat{U}\widehat{V}^{\mathbb{T}}$
  associated to a local minimizer $(\widehat{U},\widehat{V})$ of model \eqref{MC-Fnorm}
  with ${\rm rank}(\widehat{U})=r^*$ does.
 \end{remark}
 \section{An alternating MM method with extrapolation}\label{sec3}

  Fix any $(U,V)\in\mathbb{R}^{n\times r}\times\mathbb{R}^{m\times r}$.
  Since $f$ is smooth and its gradient $\nabla\!f$ is Lipschitz with modulus $L_{\!f}$,
  the function $F(\cdot,V)$ is smooth and $\nabla_{\!1}F(\cdot,V)$ is Lipschitz continuous
  with modulus $\tau_{\!V}\!:=L_{\!f}\|V\|^2$. By the decent lemma, for any $U'\in\mathbb{R}^{n\times r}$
  and $\gamma\ge\tau_{\!V}$ it holds that
  \begin{subnumcases}{}\label{FU}
   F(U',V)\le F(U,V)+\langle\nabla_{\!1}F(U,V),U'\!-\!U\rangle
    +\frac{\gamma}{2}\|U'\!-\!U\|_F^2,\\
  -F(U',V)\le -F(U,V)-\langle\nabla_{\!1}F(U,V),U'\!-\!U\rangle
    +\frac{\gamma}{2}\|U'\!-\!U\|_F^2.
  \label{nFU}
  \end{subnumcases}
  Similarly, since $F(U,\cdot)$ is a smooth function and its gradient
  $\nabla_{\!2}F(U,\cdot)$ is Lipschitz continuous with modulus
  $\tau_{U}\!:=L_{\!f}\|U\|^2$, for any $V'\in\mathbb{R}^{m\times r}$
  and $\gamma\ge\tau_{U}$ it holds that
  \begin{subnumcases}{}\label{FV}
   F(U,V')\le F(U,V)+\langle\nabla_{\!2}F(U,V),V'\!-\!V\rangle
    +\frac{\gamma}{2}\|V'\!-\!V\|_F^2,\\
   -F(U,V')\le -F(U,V)-\langle\nabla_{\!2}F(U,V),V'\!-\!V\rangle
    +\frac{\gamma}{2}\|V'\!-\!V\|_F^2.
   \label{nFV}
  \end{subnumcases}
  From inequalities \eqref{FU} and \eqref{FV}, and the expression of $\Phi_{\lambda,\mu}$, it follows that
 \begin{align*}
  \Phi_{\lambda,\mu}(U',V)
  &\le F_{U,\gamma}(U';U,V):=\langle\nabla_{\!1}F(U,V),U'\rangle+\frac{\gamma}{2}\|U'\!-\!U\|_F^2
       +\frac{\mu}{2}\|U'\|_F^2+\lambda\|U'\|_{2,0}\\
  &\qquad\qquad\qquad\qquad\quad +F(U,V)-\langle\nabla_{\!1}F(U,V),U\rangle
   +\frac{\mu}{2}\|V\|_F^2+\lambda\|V\|_{2,0},\\
  \Phi_{\lambda,\mu}(U,V')
  &\le F_{V,\gamma}(V';U,V):=\langle\nabla_{\!2}F(U,V),V'\rangle+\frac{\gamma}{2}\|V'\!-\!V\|_F^2
       +\frac{\mu}{2}\|V'\|_F^2+\lambda\|V'\|_{2,0}\\
  &\qquad\qquad\qquad\qquad\quad +F(U,V)-\langle\nabla_{\!2}F(U,V),V\rangle
   +\frac{\mu}{2}\|U\|_F^2+\lambda\|U\|_{2,0},
 \end{align*}
 which become equalities when $U'=U$ and $V'=V$. Hence, $F_{U,\gamma}(\cdot;U,V)$
 and $F_{V,\gamma}(\cdot;U,V)$ are respectively a majorization of $\Phi_{\lambda,\mu}(\cdot,V)$
 at $U$ and $\Phi_{\lambda,\mu}(U,\cdot)$ at $V$. Inspired by this, we propose
 an AMM method with extrapolation by minimizing such two majorizations in each iterate.
 \begin{algorithm}[H]
  \caption{\label{AMM}{\bf (AMM method for solving \eqref{MS-FL20})}}
  \textbf{Initialization:}
  Choose a starting point $(U^0,V^0)\in\mathbb{R}^{n\times r}\times\mathbb{R}^{m\times r}$.
  Select $\beta\in[0,1]$ and $\beta_0\in[0,\beta]$, $0<\alpha_1\le \alpha_2$.
  Let $(U^{-1},V^{-1})\!:=(U^0,V^0)$ and set $k:=0$.\\
 \textbf{while} the stopping conditions are not satisfied \textbf{do}
  \begin{itemize}
   \item[\bf 1.] Select $\gamma_{1,k}\in\tau_{\!V^{k}}+{[}\alpha_1,\alpha_2{]}$.
                 Let $\widetilde{U}^k\!:=U^k+\beta_k(U^k-U^{k-1})$ and compute
                 \begin{equation}\label{Uk-subprob}
                   U^{k+1}\in\mathop{\arg\min}_{U\in\mathbb{R}^{n\times r}}
                   \Big\{\langle\nabla_{\!1}F(\widetilde{U}^k,V^k),U\rangle
                   +\frac{\gamma_{1,k}}{2}\|U\!-\!\widetilde{U}^k\|_F^2 +\frac{\mu}{2}\|U\|_F^2+\lambda\|U\|_{2,0}\Big\}.
                 \end{equation}

 \item[\bf 2.]  Select $\gamma_{2,k}\in\tau_{\!U^{k+1}}+{[}\alpha_1,\alpha_2{]}$.
                Let $\widetilde{V}^k\!:=V^k+\beta_k(V^k-V^{k-1})$ and compute
                \begin{equation}\label{Vk-subprob}
                 V^{k+1}\!\in\mathop{\arg\min}_{\!V\in\mathbb{R}^{m\times r}}
                 \Big\{\langle\nabla_{\!2}F(U^{k+1},\widetilde{V}^k),V\rangle
                  +\frac{\gamma_{2,k}}{2}\|V\!-\!\widetilde{V}^k\|_F^2+\frac{\mu}{2}\|V\|_F^2 +\lambda\|V\|_{2,0}\Big\}.
                \end{equation}

 \item[\bf 3.] Update $\beta_k$ by $\beta_{k+1}\in[0,\beta]$ and let $k\leftarrow k+1$.
 \end{itemize}
 \textbf{end while}
 \end{algorithm}
 \begin{remark}\label{remark1-AMM}
  {\bf(a)} Algorithm \ref{AMM} is a special case of the iPALM in \cite{Pock16}
  with $\alpha_i^k\!=\!\beta_i^k$ for $i=1,2$, but as will be shown below
  our global convergence analysis is different from that of \cite{Pock16} since,
  the boundedness of the generated sequence $\{(U^k,V^k)\}$ is directly achieved
  under a mild restriction on $\beta$ by leveraging the structure of $F$,
  and moreover, a quantification on $\beta$ is also provided.

  \noindent
  {\bf(b)} Let $G^k=\frac{1}{\mu+\gamma_{1,k}}(\gamma_{1,k}\widetilde{U}^k\!-\!\nabla_{\!1}F(\widetilde{U}^k,V^k))$
  and $H^k=\frac{1}{\mu+\gamma_{2,k}}(\gamma_{2,k}\widetilde{V}^k\!-\!\nabla_{\!2}F(U^{k+1},\widetilde{V}^k))$.
  By the expression of $F$, the columns of $U^{k+1}$ and $V^{k+1}$ have the following closed form:
  \begin{align*}
    U_i^{k+1}&={\rm sign}\Big[\max\Big(0,\|G_i^{k}\|\!-\!\sqrt{2(\mu\!+\!\gamma_{1,k})^{-1}\lambda}\Big)\Big]G_i^{k}
    \ \ {\rm for}\ i=1,\ldots,r;\\
    V_i^{k+1}&={\rm sign}\Big[\max\Big(0,\|H_i^{k}\|\!-\!\sqrt{2(\mu\!+\!\gamma_{2,k})^{-1}\lambda}\Big)\Big]H_i^{k}
    \ \ {\rm for}\ i=1,\ldots,r.
  \end{align*}
  Consequently, we deduce that each step of Algorithm \ref{AMM} involves about $4mnr$ flops.
%
%
 \end{remark}

 Next we shall establish the global convergence of Algorithm \ref{AMM} by following
 the analysis recipe of algorithms for nonconvex nonsmooth problems
 in the KL framework (see \cite{Attouch10,Botle14,LiuPong191,Pock16}). Define
 \begin{equation}\label{alphak}
  \alpha_{1,k}\!:=\gamma_{1,k}-\tau_{V^{k}}\ \ {\rm and}\ \
  \alpha_{2,k}\!:=\gamma_{2,k}-\tau_{U^{k+1}}\ \ {\rm for\ each}\ k\in\mathbb{N}.
 \end{equation}
 The following proposition characterizes an important property of
 the sequence $\{(U^k,V^k)\}_{k\in\mathbb{N}}$, whose proof is included in
 Appendix B.
 \begin{proposition}\label{prop1-UVk}
  Let $\{(U^k,V^k)\}_{k\in\mathbb{N}}$ be the sequence generated by Algorithm \ref{AMM}.
  Then, for any given $\rho_1\in(0,1)$ and $\rho_2\in(0,1)$, the following inequality
  holds for each $k\in\mathbb{N}$:
  \begin{align}\label{descent-ineq1}
   &\Big[\Phi_{\lambda,\mu}(U^{k+1},V^{k+1})+\frac{\rho_1\alpha_{1,k}}{2}\|U^{k+1}-U^k\|_F^2
     +\frac{\rho_2\alpha_{2,k}}{2}\|V^{k+1}-V^k\|_F^2\Big]\nonumber\\
   &-\Big[\Phi_{\lambda,\mu}(U^{k},V^{k})+\frac{\rho_1\alpha_{1,k}}{2}\|U^{k}-U^{k-1}\|_F^2
     +\frac{\rho_2\alpha_{2,k}}{2}\|V^{k}-V^{k-1}\|_F^2\Big]\nonumber\\
   &\le-\Big[\frac{\rho_1\alpha_{1,k}}{2}-\frac{(2(1\!-\!\rho_1)\tau_{V^{k}}+\alpha_{1,k})\beta_k^2}{2(1\!-\!\rho_1)}\Big]
     \big\|U^{k}\!-\!U^{k-1}\big\|_F^2\nonumber\\
   &\quad-\Big[\frac{\rho_2\alpha_{2,k}}{2}-\frac{(2(1\!-\!\rho_2)\tau_{U^{k+1}}+\alpha_{2,k})\beta_k^2}{2(1\!-\!\rho_2)}\Big]
     \big\|V^{k}\!-\!V^{k-1}\big\|_F^2.
  \end{align}
  Consequently, $(U^k,V^k)\in\mathcal{L}_{\lambda,\mu}\!:=\big\{(U,V)\in\mathbb{R}^{n\times r}\times\mathbb{R}^{m\times r}\,|\,
 \Phi_{\lambda,\mu}(U,V)\le\Phi_{\lambda,\mu}(U^0,V^0)\big\}$ whenever
  $\beta_k\!\in[0,\min(\overline{\beta}_{1,k},\overline{\beta}_{2,k})]$ with
  \(
    \overline{\beta}_{1,k}\!:=\!\sqrt{\frac{\rho_1(1-\rho_1)(\gamma_{1,k}-\tau_{\!V^{k}})}
    {2(1-\rho_1)\tau_{\!V^{k}}+(\gamma_{1,k}-\tau_{\!V^{k}})}}
  \)
  and
  \(
   \overline{\beta}_{2,k}\!:=\!\sqrt{\frac{\rho_2(1-\rho_2)(\gamma_{2,k}-\tau_{\!U^{k+1}})}
   {2(1-\rho_2)\tau_{\!U^{k+1}}+(\gamma_{2,k}-\tau_{\!U^{k+1}})}}.
  \)
 \end{proposition}

 \begin{remark}\label{remark-alpha}
 {\bf(a)} Write $\overline{\beta}:=\inf_{k\in\mathbb{N}}\min(\overline{\beta}_{1,k},\overline{\beta}_{2,k})$.
 Clearly, $\overline{\beta}$ is well defined. By the second part of Proposition \ref{prop1-UVk},
 if $\beta$ is chosen from the interval $[0,\overline{\beta}]$,
 then $\tau\!:=\sup_{k\in\mathbb{N}}\max(\tau_{U^k},\tau_{V^k})<\infty$
 is well defined. We see that, when taking $\gamma_{1,k}=\eta_1\tau_{\!V^{k}}$
 and $\gamma_{2,k}=\eta_2\tau_{\!U^{k+1}}$ for $\eta_1=2$ and $\eta_2=2$,
 the value of $\overline{\beta}$ equals
 \(
   \min\Big(\!\sqrt{\frac{\rho_1(1-\rho_1)}
   {2(1-\rho_1)+1}},\sqrt{\frac{\rho_2(1-\rho_2)}{2(1-\rho_2)+1}}\Big),
  \)
  whose maximum is close to $0.366$.

  \noindent
  {\bf(b)} When $f$ is convex, $\frac{\tau_{\!V^{k}}}{2}\|U^{k}\!-\!\widetilde{U}^k\|_F^2$
  in \eqref{convex-monotone1} and $\frac{\tau_{U^{k+1}}}{2}\|V^{k}\!-\!\widetilde{V}^k\|_F^2$
  in \eqref{FV-ineq1} do not appear. Now Proposition \ref{prop1-UVk} holds with
  \(
    \overline{\beta}_{1,k}\!:=\!\sqrt{\frac{\rho_1(1-\rho_1)(\gamma_{1,k}-\tau_{\!V^{k}})}
    {(1-\rho_1)\tau_{\!V^{k}}+(\gamma_{1,k}-\tau_{V^{k}})}}
  \)
  and
  \(
   \overline{\beta}_{2,k}\!:=\!\sqrt{\frac{\rho_2(1-\rho_2)(\gamma_{2,k}-\tau_{\!U^{k+1}})}
   {(1-\rho_2)\tau_{\!U^{k+1}}+(\gamma_{2,k}-\tau_{\!U^{k+1}})}}.
  \)
 \end{remark}

 To achieve the global convergence of Algorithm \ref{AMM},
  with some $\rho_1,\rho_2\in(0,\frac{\alpha_1}{2\alpha_2})$ we define
 \begin{equation}
  \Xi_{\lambda,\mu}(U,V,U',V')
  :=\Phi_{\lambda,\mu}(U,V)+\frac{\rho_1\alpha_2}{2}\|U-U'\|_F^2+\frac{\rho_2\alpha_2}{2}\|V-V'\|_F^2.
 \end{equation}
 The following proposition characterizes the properties of the potential function $\Xi_{\lambda,\mu}$
 on the sequence $\{(U^k,V^k,U^{k-1},V^{k-1})\}_{k\in\mathbb{N}}$, whose proof is included in Appendix C.
 \begin{proposition}\label{prop2-UVk}
  Let $\{(U^k,V^k)\}_{k\in\mathbb{N}}$ be the sequence generated by Algorithm \ref{AMM}
  with $\beta\in[0,\overline{\beta}]$, where $\overline{\beta}$ is the constant
  defined in Remark \ref{remark-alpha} (a). Then, the following statements hold.
  \begin{itemize}
   \item [(i)] With $\nu_{1,k}=\frac{(\alpha_{1,k}-\rho_1\alpha_2)(\rho_1\alpha_2-2\tau_{\!V^{k}}\beta_k^2)-\alpha_{1,k}^2\beta_k^2}
               {2(\alpha_{1,k}-\rho_1\alpha_2)}$ and $\nu_{2,k}=\frac{(\alpha_{2,k}-\rho_2\alpha_2)(\rho_2\alpha_2-2\tau_{\!U^{k+1}}\beta_k^2)-\alpha_{2,k}^2\beta_k^2}
               {2(\alpha_{2,k}-\rho_2\alpha_2)}$,
               \begin{align*}
                &\Xi_{\lambda,\mu}(U^{k+1},V^{k+1},U^k,V^k)
                  -\Xi_{\lambda,\mu}(U^k,V^k,U^{k-1},V^{k-1})\\
                 &\le -\nu_{1,k}\|U^{k}-U^{k-1}\|_F^2-\nu_{2,k}\|V^k-V^{k-1}\|_F^2
                 \quad\ {\rm for\ each}\ k\in\mathbb{N}.
                \end{align*}

   \item [(ii)] The sequence $\{(U^k,V^k)\}_{k\in\mathbb{N}}$ is bounded.
                Therefore, the set of accumulation points of the sequence
                $\{(U^{k},V^{k},U^{k-1},V^{k-1})\}_{k\in\mathbb{N}}$, denoted by $\Upsilon$,
                is nonempty and compact.

   \item[(iii)] If $\beta\in[0,\min(\overline{\beta},\widetilde{\beta})]$ with $0\leq\widetilde{\beta}<\min\Big(\sqrt{\frac{\rho_1(1-\rho_1)\alpha_2}{2(1-\rho_1)\tau+\alpha_2}},
                \sqrt{\frac{\rho_2(1-\rho_2)\alpha_2}{2(1-\rho_2)\tau+\alpha_2}}\Big)$ for
                $\tau$ in Remark \ref{remark-alpha} (a),
                then $\{\Xi_{\lambda,\mu}(U^{k}\!,V^{k}\!,U^{k-1}\!\!,V^{k-1}\!)\}_{k\in\mathbb{N}}$
                has a limit as $k\!\to\!\infty$, say $\varpi^*$, and $\Xi_{\lambda,\mu}\!\equiv\!\varpi^*$
                on $\Upsilon$.

   \item [(iv)] If $\beta\in[0,\min(\overline{\beta},\widetilde{\beta})]$ where $\widetilde{\beta}$
                is same as in part (iii), then for each $k\in\mathbb{N}$ it holds that
               \begin{align*}
                {\rm dist}\big(0,\partial\Xi_{\lambda,\mu}(U^{k+1},V^{k+1},U^k,V^k)\big)
                &\leq c_1\big(\|U^{k+1}\!-\!U^{k}\!\|_F\!+\|U^k-U^{k-1}\|_F\big)\\
                &\quad+c_2\big(\|V^{k+1}\!-\!V^{k}\!\|_F^2\!+\|V^k-V^{k-1}\|_F\big)
               \end{align*}
               for $c_1\!=\tau+\overline{\gamma}+2\rho_1\alpha_2$ and $
               c_2\!=c_f+2\tau+\overline{\gamma}+2\rho_2\alpha_2$ with
               $c_f\!=\sup_{k\in\mathbb{N}}\{\|\nabla\!f(U^k(V^k)^{\mathbb{T}})\|\}$.
 \end{itemize}
 \end{proposition}

 \begin{remark}\label{remark2-AMM}
  By Remark \ref{remark-alpha} (b), the constants $\overline{\beta}$ and
  $\widetilde{\beta}$ in Proposition \ref{prop2-UVk} can be improved when $f$ is convex.
  From equation \eqref{gradXi-UV1} and Proposition \ref{subdiff-Phi},
  whenever $(\overline{U},\overline{V},\overline{U},\overline{V})
  \in{\rm crit}\,\Xi_{\lambda,\mu}$, we have $(\overline{U},\overline{V})\in
  {\rm crit}\Phi_{\lambda,\mu}$. Along with Proposition \ref{prop2-UVk} (iv),
  if the sequence $\{(U^k,V^k)\}_{k\in\mathbb{N}}$ generated by Algorithm \ref{AMM}
  with $\beta\in[0,\min(\overline{\beta},\widetilde{\beta})]$ is convergent,
  then its limit is a critical point of $\Phi_{\lambda,\mu}$.
 \end{remark}

  Since the zero-norm and the function $\theta(Z)\!:=(\|Z_1\|,\ldots,\|Z_{r}\|)$
  for $Z\in\mathbb{R}^{n\times r}$ are semialgebraic, the column $\ell_{2,0}$-norm,
  as a composition of $\theta$ and the zero-norm, is semialgebraic.
  This means that $\Xi_{\lambda,\mu}$ is a KL function (see \cite[Section 4]{Attouch10}).
  By Proposition \ref{prop2-UVk} and Remark \ref{remark2-AMM},
  using the same arguments as those for \cite[Theorem 3.2]{Attouch10}
  or \cite[Theorem 3.1]{LiuPong191} yields the following result.
 \begin{theorem}\label{theorem-AMM}
  Let $\{(U^k,V^k)\}_{k\in\mathbb{N}}$ be the sequence given by Algorithm \ref{AMM}
  with $\beta\in[0,\min(\overline{\beta},\widetilde{\beta})]$
  for solving problem \eqref{MS-FL20} associated to $\lambda$ and $\mu$.
  Then, the sequence $\{(U^k,V^k)\}_{k\in\mathbb{N}}$ is convergent
  and its limit, say $(U^*,V^*)$, is a critical point of $\Phi_{\lambda,\mu}$,
  which by Proposition \ref{prop2-Phi} is also a local optimizer of
  problem \eqref{MS-FL20} if $(U^*,V^*)$ is a local minimizer of $F_{\mu}$.
 \end{theorem}
 \section{A hybrid alternating MM method}\label{sec4}

  Algorithm \ref{AMM} is actually a majorized alternating proximal (MAP) method for
  solving \eqref{MS-FL20}. Indeed, for any $(U,V),(G,H)\in\mathbb{R}^{n\times r}\times \mathbb{R}^{m\times r}$,
  we have
 \begin{equation*}
  f(UV^\mathbb{T})\le \widehat{F}(U,V,G,H):=
  f(GH^\mathbb{T})+\langle\nabla\!f(GH^\mathbb{T}),UV^\mathbb{T}\!-\!GH^\mathbb{T}\rangle
  +\frac{L_{\!f}}{2}\|UV^\mathbb{T}\!-\!GH^\mathbb{T}\|_F^2,
 \end{equation*}
 which by the expression of $\Phi_{\lambda,\mu}$ implies that
 \begin{align*}
  \Phi_{\lambda,\mu}(U,V)\le\widehat{\Phi}_{\lambda,\mu}(U,V,G,H):=
  \widehat{F}(U,V,G,H)+\frac{\mu}{2}\big(\|U\|_F^2\!+\!\|V\|_F^2\big)+\lambda\big(\|U\|_{2,0}\!+\!\|V\|_{2,0}\big).
 \end{align*}
 This, along with $\Phi_{\lambda,\mu}(G,H)=\widehat{\Phi}_{\lambda,\mu}(G,H,G,H)$,
 means that $\widehat{\Phi}_{\lambda,\mu}(\cdot,\cdot,G,H)$ is a majorization of
 $\Phi_{\lambda,\mu}$ at $(G,H)$. Observe that the subproblems \eqref{Uk-subprob}
 and \eqref{Vk-subprob} are respectively equivalent to
 \begin{subnumcases}{}\label{EUk-subprob}
  U^{k+1}\in\mathop{\arg\min}_{U\in \mathbb{R}^{n\times r}}
         \Big\{\widehat{\Phi}_{\lambda,\mu}(U,V^k,\widetilde{U}^k,V^k)
         +\frac{1}{2}\|U-\widetilde{U}^k\|_{\mathcal{A}_{k}}^2\Big\},\\
  V^{k+1}\in\mathop{\arg\min}_{V\in \mathbb{R}^{m\times r}}
          \Big\{\widehat{\Phi}_{\lambda,\mu}({U}^{k+1},V,{U}^{k+1},\widetilde{V}^{k})
          +\frac{1}{2}\|V-\widetilde{V}^{k}\|_{\mathcal{B}_{k+1}}^2\Big\}
 \label{EVk-subprob}
 \end{subnumcases}
 where $\mathcal{A}_{k}(X)\!:=\!X(\gamma_{1,k}I\!-\!L_{\!f}(V^k)^{\mathbb{T}}V^k)$
 for $X\in\mathbb{R}^{n\times r}$
 and $\mathcal{B}_{k}(Z)\!:=\!Z(\gamma_{2,k-1}I\!-\!L_{\!f}(U^k)^{\mathbb{T}}U^k)$
 for $Z\in\mathbb{R}^{m\times r}$ are the self-adjoint positive definite linear operators,
 and the proximal terms $\frac{1}{2}\|U\!-\widetilde{U}^k\|_{\mathcal{A}_{k}}^2$
 and $\frac{1}{2}\|V\!-\widetilde{V}^{k}\|_{\mathcal{B}_{k+1}}^2$ are introduced to
 ensure that the subproblems have a closed-form solution.
 Hence, Algorithm \ref{AMM} is precisely minimizing $\widehat{\Phi}_{\lambda,\mu}(U,V,G,H)$
 in an alternating proximal way.
 Next we develop another MAP method by minimizing $\widehat{\Phi}_{\lambda,\mu}$.
 Its iterates are described as follows.
 \begin{algorithm}[h]
  \caption{\label{MAPM}{\bf (MAP method for solving \eqref{MS-FL20})}}
  \textbf{Initialization:} Select the parameters $\varrho\in(0,1),\underline{\gamma_1}>0,
  \underline{\gamma_2}>0,\gamma_{1,0}>0$ and $\gamma_{2,0}>0$.
  Choose $P^0\!\in\mathbb{O}^{m\times r},Q^0\!\in\mathbb{O}^{n\times r},D^0=I_{r}$.
  Let $\overline{U}^0=Q^0$ and $\overline{V}^0=P^0$. Set $k:=0$.\\
 \textbf{while} the stopping conditions are not satisfied \textbf{do}
  \begin{itemize}
   \item[\bf 1.] Compute
                 \(
                  U^{k+1}\in\displaystyle{\mathop{\arg\min}_{U\in\mathbb{R}^{n\times r}}}
                 \Big\{\widehat{\Phi}_{\lambda,\mu}(U,\overline{V}^{k},\overline{U}^k,\overline{V}^k)          +\frac{\gamma_{1,k}}{2}\|U-\overline{U}^k\|_F^2\Big\}.
                 \)

   \item[\bf 2.] Perform an SVD for $U^{k+1}D^k$ such that $U^{k+1}D^k=\widehat{P}^{k+1}(\widehat{D}^{k+1})^2(\widehat{Q}^{k+1})^{\mathbb{T}}$,
                 and set
                 \[
                  \widehat{U}^{k+1}:=\widehat{P}^{k+1}\widehat{D}^{k+1}\ \ {\rm and}\ \
                   \widehat{V}^{k+1}\!:=P^k\widehat{Q}^{k+1}\widehat{D}^{k+1}.
                 \]

 \item[\bf 3.]  Compute
                \(
                  V^{k+1}\in\displaystyle{\mathop{\arg\min}_{V\in \mathbb{R}^{m\times r}}}
                   \Big\{\widehat{\Phi}_{\lambda,\mu}(\widehat{U}^{k+1},V,\widehat{U}^{k+1},\widehat{V}^{k+1})
                    +\frac{\gamma_{2,k}}{2}\|V-\widehat{V}^{k+1}\|_F^2\Big\}.
                \)

 \item[\bf 4.] Perform an SVD for $V^{k+1}\widehat{D}^{k+1}$ such that
                 $V^{k+1}\widehat{D}^{k+1}=P^{k+1}(D^{k+1})^2(Q^{k+1})^{\mathbb{T}}$, and set
               \[
                 \overline{U}^{k+1}\!:=\widehat{P}^{k+1}Q^{k+1}D^{k+1}\ \ {\rm and}\ \
                 \overline{V}^{k+1}\!:=P^{k+1}D^{k+1}.
               \]

 \item[\bf 5.] Set $\gamma_{1,k+1}=\max(\underline{\gamma_1},\varrho \gamma_{1,k})$ and
               $\gamma_{2,k+1}=\max(\underline{\gamma_2},\varrho \gamma_{2,k})$.
               Let $k\leftarrow k+1$.
 \end{itemize}
 \textbf{end while}
 \end{algorithm}
 \begin{remark}\label{remark-MAPM}
  {\bf(a)} For each $k\in\mathbb{N}$, let $\overline{X}^{k+1}\!:=U^{k+1}(\overline{V}^{k})^{\mathbb{T}}\!=U^{k+1}D^k(P^k)^{\mathbb{T}}$.
  Since $P^k\in\mathbb{O}^{m\times r}$, Step 2 is actually performing an SVD
  of $\overline{X}^{k+1}$ to seek a new factor pair $(\widehat{U}^{k+1},\widehat{V}^{k+1})$
  such that the subproblem in Step 3 has a closed-form solution. As will be shown in \eqref{good-equa}
  later, $(\widehat{U}^{k+1},\widehat{V}^{k+1})$ is at least as good as $(U^{k+1},\overline{V}^{k})$
  for the function $\widehat{\Phi}_{\lambda,\mu}(\cdot,\cdot,\overline{U}^k,\overline{V}^k)$.
  Similarly, by letting $\widehat{X}^{k+1}\!:=\widehat{U}^{k+1}(V^{k+1})^{\mathbb{T}}
  =\widehat{P}^{k+1}\widehat{D}^{k+1}(V^{k+1})^{\mathbb{T}}$,
  Step 4 is performing an SVD of $\widehat{X}^{k+1}$ to seek a factor pair
  $(\overline{U}^{k+1},\overline{V}^{k+1})$ such that the subproblem in Step 1
  has a closed-form solution. To the best of our knowledge, such a technique
  appeared in the alternating least squares method of \cite{Hastie15}.

  \noindent
 {\bf(b)} For each $k\!\in\mathbb{N}$, let $G^k\!:=\!\big(L_{\!f}\widehat{Z}^kP^k\!+\!\gamma_{1,k}\widehat{P}^kQ^k\big)D^k(\Lambda^k)^{-1}$
 with $\widehat{Z}^k\!:=\!\widehat{X}^k\!-\!L_{\!f}^{-1}\nabla\!f(\widehat{X}^k)$ for $\widehat{P}^0=I$,
 and $\Lambda^k\!:=\!\big[L_{\!f}(D^k)^2\!+\!(\mu+\!\gamma_{1,k})I_r\big]^{1/2}$.
 By the expression of $\widehat{\Phi}_{\lambda,\mu}$, Step 1 is equivalent to
  \[
    U^{k+1}\in\mathop{\arg\min}_{U\in \mathbb{R}^{n\times r}}
   \Big\{\frac{1}{2}\big\|G^k-U\Lambda^k\big\|_F^2+\lambda\|U\|_{2,0}\Big\}.
  \]
  By this, it is easy to calculate that the columns of $U^{k+1}$ take the following form
  \begin{equation}\label{partU-equal1}
    U_i^{k+1}=\frac{{\rm sign}\big[\max(0,\|G_i^{k}\|\!-\!\sqrt{2\lambda})\big]}{\sigma_i(\Lambda^k)}G_i^{k}
    \ \ {\rm for}\ i=1,\ldots,r.
  \end{equation}
  Similarly, by letting $\Delta^{k+1}:=\!\big[L_{\!f}(\widehat{D}^{k+1})^2+(\mu+\gamma_{2,k})I_r\big]^{1/2}$,
  $\overline{Z}^{k+1}\!:=\overline{X}^{k+1}\!-\!L_{\!f}^{-1}\nabla\!f(\overline{X}^{k+1})$
  and $H^{k+1}\!:=\!\big(L_{\!f}(\overline{Z}^{k+1})^{\mathbb{T}}\widehat{P}^{k+1}\!+\!\gamma_{2,k}P^k\widehat{Q}^{k+1}\big)
  \widehat{D}^{k+1}(\Delta^{k+1})^{-1}$ for $k\!\in\mathbb{N}$, Step 3 is equivalent to seeking
  \[
    V^{k+1}\in\mathop{\arg\min}_{V\in \mathbb{R}^{m\times r}}
   \Big\{\frac{1}{2}\big\|H^{k+1}-V\Delta^{k+1}\big\|_F^2+\lambda\|V\|_{2,0}\Big\},
  \]
  which implies that the columns of the matrix $V^{k+1}$ take the following form
  \begin{equation}\label{partV-equal1}
   V_i^{k+1}=\frac{{\rm sign}\big[\max(0,\|H_i^{k+1}\|\!-\!\sqrt{2\lambda})\big]}{\sigma_i(\Delta^{k+1})}H_i^{k+1}
   \ \ {\rm for}\ \ i=1,2,\ldots,r.
  \end{equation}
  Thus, we deduce that each step of Algorithm \ref{MAPM} involves about $4mnr+2(m+n)r^2$ flops.
 \end{remark}

 The following proposition states the properties of the sequence
 generated by Algorithm \ref{MAPM}.
 \begin{proposition}\label{prop1-MAPM}
  Let $\big\{(U^k,V^k,\widehat{U}^{k},\widehat{V}^{k},
  \overline{U}^{k},\overline{V}^{k})\big\}_{k\in\mathbb{N}}$
  be generated by Algorithm \ref{MAPM}. Then,
  \begin{itemize}
  \item[(i)] for each $k\in\mathbb{N}$, it holds that
             \begin{align}\label{descrease-ineq1}
              {\Phi}_{\lambda,\mu}(\overline{U}^{k},\overline{V}^{k})
               &\ge {\Phi}_{\lambda,\mu}(\widehat{U}^{k+1},\widehat{V}^{k+1})
                 +\frac{\gamma_{1,k}}{2}\|U^{k+1}-\overline{U}^k\|_F^2\nonumber\\
               &\ge{\Phi}_{\lambda,\mu}(\overline{U}^{k+1},\overline{V}^{k+1})
                +\frac{\gamma_{1,k}}{2}\|U^{k+1}\!-\overline{U}^k\|_F^2
                +\frac{\gamma_{2,k}}{2}\|V^{k+1}\!-\widehat{V}^{k+1}\|_F^2,\nonumber
             \end{align}
            and hence $\{\Phi_{\lambda,\mu}(\overline{U}^k,\overline{V}^k)\}_{k\in\mathbb{N}}$
            and $\{\Phi_{\lambda,\mu}(\widehat{U}^k,\widehat{V}^k)\}_{k\in\mathbb{N}}$
            are nonincreasing and convergent;

  \item[(ii)] the sequence $\big\{(U^k,V^k,\widehat{U}^{k},\widehat{V}^{k},
               \overline{U}^{k},\overline{V}^{k})\big\}_{k\in\mathbb{N}}$ is bounded;

  \item [(iii)] there exists $\overline{k}\in\!\mathbb{N}$ such that for all $k\ge\!\overline{k}$,
               \(
                 J_{V^{k}}\!=\!J_{U^{k}}\!=\!J_{\widehat{U}^{k}}\!=\!J_{\widehat{V}^{k}}
               \!=\!J_{\overline{V}^{k}}\!=\!J_{\overline{U}^{k}}\!=\!J_{\overline{U}^{k+1}}.
               \)
  \end{itemize}
 \end{proposition}
 \begin{proof}
  {\bf(i)} By using $\Phi_{\lambda,\mu}(\overline{U}^{k},\overline{V}^{k})\!=\!
  \widehat{\Phi}_{\lambda,\mu}(\overline{U}^{k},\overline{V}^{k},\overline{U}^{k},\overline{V}^{k})$
  and the definitions of $U^{k+1}$ and $ V^{k+1}$,
 \begin{subequations}
 \begin{align}\label{Fval-U}
  {\Phi}_{\lambda,\mu}(\overline{U}^{k},\overline{V}^{k})
  \ge\widehat{\Phi}_{\lambda,\mu}(U^{k+1},\overline{V}^k,\overline{U}^{k},\overline{V}^{k})
   +\frac{\gamma_{1,k}}{2}\|U^{k+1}-\overline{U}^k\|_F^2;\qquad\quad\\
   \label{Fval-V}
  {\Phi}_{\lambda,\mu}(\widehat{U}^{k+1},\widehat{V}^{k+1})
 \ge\widehat{\Phi}_{\lambda,\mu}\big(\widehat{U}^{k+1},V^{k+1},\widehat{U}^{k+1},\widehat{V}^{k+1}\big)
   +\frac{\gamma_{2,k}}{2}\|V^{k+1}-\widehat{V}^{k+1}\|_F^2.
 \end{align}
 \end{subequations}
 By Remark \ref{remark-MAPM} (a) and Step 2,
 $\overline{X}^{k+1}=U^{k+1}(\overline{V}^k)^\mathbb{T}=\widehat{U}^{k+1}(\widehat{V}^{k+1})^\mathbb{T}$,
 which implies that
 \begin{equation}\label{Fequa1}
    \widehat{F}(U^{k+1},\overline{V}^k,\overline{U}^{k},\overline{V}^{k})
  =\widehat{F}(\widehat{U}^{k+1},\widehat{V}^{k+1},\overline{U}^{k},\overline{V}^{k}).
 \end{equation}
 In addition, by the definitions of $\widehat{U}^{k+1}$ and $\widehat{V}^{k+1}$,
 equation \eqref{rank-chara} and \cite[Lemma 1]{Srebro05},
 \begin{subnumcases}{}
   \frac{1}{2}\big(\|U^{k+1}\|_F^2+\|\overline{V}^k\|_F^2\big)
   \geq\|\overline{X}^{k+1}\|_*
    = \frac{1}{2}\big(\|\widehat{U}^{k+1}\|_F^2+\|\widehat{V}^{k+1}\|_F^2\big);\nonumber\\
    \frac{1}{2}\big(\|U^{k+1}\|_{2,0}+\|\overline{V}^k\|_{2,0}\big)
    \geq{\rm rank}(\overline{X}^{k+1})
    = \frac{1}{2}\big(\|\widehat{U}^{k+1}\|_{2,0}+\|\widehat{V}^{k+1}\|_{2,0}\big).\nonumber
 \end{subnumcases}
 By combining the two inequalities with equality \eqref{Fequa1}, it is immediate to obtain that
 \begin{equation}\label{good-equa}
  \widehat{\Phi}_{\lambda,\mu}(U^{k+1},\overline{V}^k,\overline{U}^{k},\overline{V}^{k})
  \ge\widehat{\Phi}_{\lambda,\mu}(\widehat{U}^{k+1},\widehat{V}^{k+1},\overline{U}^{k},\overline{V}^{k}).
 \end{equation}
 Similarly, by Remark \ref{remark-MAPM} (a) and Step 3,
 $\widehat{X}^{k+1}=\widehat{U}^{k+1}(V^{k+1})^\mathbb{T}=\overline{U}^{k+1}(\overline{V}^{k+1})^\mathbb{T}$,
 which along with the definitions of $\overline{U}^{k+1}$ and $\overline{V}^{k+1}$
 implies that the following inequality holds:
 \[
   \widehat{\Phi}_{\lambda,\mu}(\widehat{U}^{k+1},V^{k+1},\widehat{U}^{k+1},\widehat{V}^{k+1})
   \ge\widehat{\Phi}_{\lambda,\mu}(\overline{U}^{k+1},\overline{V}^{k+1},\widehat{U}^{k+1},\widehat{V}^{k+1}).
 \]
 Now substituting the last two inequalities into \eqref{Fval-U} and \eqref{Fval-V}
 respectively yields that
  \begin{subequations}
  \begin{align}\label{Fval-equaU}
   {\Phi}_{\lambda,\mu}(\overline{U}^{k},\overline{V}^{k})
    \geq\widehat{\Phi}_{\lambda,\mu}(\widehat{U}^{k+1},\widehat{V}^{k+1},\overline{U}^{k},\overline{V}^{k})
      +\frac{\gamma_{1,k}}{2}\|U^{k+1}-\overline{U}^k\|_F^2;\qquad\\
   \label{Fval-equaV}
  {\Phi}_{\lambda,\mu}(\widehat{U}^{k+1},\widehat{V}^{k+1})
    \geq\widehat{\Phi}_{\lambda,\mu}(\overline{U}^{k+1},\overline{V}^{k+1},\widehat{U}^{k+1},\widehat{V}^{k+1})
     +\frac{\gamma_{2,k}}{2}\|V^{k+1}-\widehat{V}^{k+1}\|_F^2.
  \end{align}
  \end{subequations}
  In addition, by the definition of $F$ and $\widehat{F}$, we have
  $F(\widehat{U}^{k+1}\!,\widehat{V}^{k+1})\!\leq \widehat{F}(\widehat{U}^{k+1}\!,\widehat{V}^{k+1},\overline{U}^{k}\!,\overline{V}^{k})$,
  and $\widehat{\Phi}_{\lambda,\mu}(\widehat{U}^{{k}+1},\widehat{V}^{k+1},\overline{U}^{k},\overline{V}^{k})
  \ge{\Phi}_{\lambda,\mu}(\widehat{U}^{k+1},\widehat{V}^{k+1})$.
  Along with \eqref{Fval-equaU}, we get the first inequality in (i).
  From the first inequality, inequality \eqref{Fval-equaV}
  and $\widehat{\Phi}_{\lambda,\mu}(\overline{U}^{k+1},\overline{V}^{k+1},\widehat{U}^{k+1},\widehat{V}^{k+1})
  \ge{\Phi}_{\lambda,\mu}(\overline{U}^{k+1},\overline{V}^{k+1})$,
  we obtain the  second inequality of part (i).

  \noindent
  {\bf(ii)} From Step 5 of Algorithm \ref{MAPM}, $\gamma_{1,k}\ge\underline{\gamma_1}$
  and $\gamma_{2,k}\ge\underline{\gamma_2}$. Together with part (i), for each $k\in\mathbb{N}$,
  \begin{align*}
   {\Phi}_{\lambda,\mu}(\overline{U}^{0},\overline{V}^{0})
   &\ge{\Phi}_{\lambda,\mu}(\widehat{U}^{1},\widehat{V}^{1})
   \ge{\Phi}_{\lambda,\mu}(\overline{U}^{1},\overline{V}^{1})
   \ge\cdots\\
   &\ge{\Phi}_{\lambda,\mu}(\overline{U}^{k-1},\overline{V}^{k-1})
   \ge {\Phi}_{\lambda,\mu}(\widehat{U}^{k},\widehat{V}^{k})
   \ge {\Phi}_{\lambda,\mu}(\overline{U}^{k},\overline{V}^{k}).
  \end{align*}
  Recall that the function $\Phi_{\lambda,\mu}$ is coercive.
  So, the sequence $\{(\overline{U}^{k},\overline{V}^{k},\widehat{U}^{k},\widehat{V}^{k})\}_{k\in\mathbb{N}}$
  is bounded. Together with part (i), it follows that
  the sequence $\{(U^k,V^k)\}_{k\in\mathbb{N}}$ is also bounded.

  \noindent
  {\bf(iii)} Fix an arbitrary $k\in\mathbb{N}$. We first argue that the following inclusions hold:
  \begin{align}\label{UVk-indxset}
   J_{\overline{U}^{k+1}}\subseteq J_{\widehat{U}^{k+1}}\subseteq J_{\overline{U}^{k}},\,
   J_{V^{k+1}}\subseteq J_{\widehat{V}^{k+1}}=J_{\widehat{U}^{k+1}}
   \subseteq J_{\overline{V}^{k}}\ \ {\rm and}\ \
   J_{{U}^{k+1}}\subseteq J_{\overline{U}^{k}}.
 \end{align}
  By the definitions of $(\overline{U}^k,\overline{V}^k)$ and
  $(\widehat{U}^k,\widehat{V}^k)$, it is easy to check that
  $J_{\overline{U}^{k}}=J_{\overline{V}^{k}}$ and
  $J_{\widehat{U}^{k}}=J_{\widehat{V}^{k}}$.
  By \eqref{partU-equal1}, $J_{{U}^{k+1}}\subseteq J_{G^{k}}$,
  while by the expression of $G^k$ in Remark \ref{remark-MAPM} (b), we deduce that
  $J_{G^{k}}\subseteq J_{D^k}$. This, by $\overline{V}^k=P^kD^k$, implies that
  $J_{{U}^{k+1}}\subseteq J_{\overline{V}^k}=J_{\overline{U}^{k}}$.
  So, the last inclusion in \eqref{UVk-indxset} holds.
  By the expression of $V^{k+1}$ in \eqref{partV-equal1}, we deduce that
  $J_{V^{k+1}}\subseteq J_{\widehat{U}^{k+1}}=J_{\widehat{V}^{k+1}}$.
  Together with $J_{U^{k+1}}\subseteq J_{\overline{V}^{k}}$,
  \[
    \|\widehat{U}^{k+1}\|_{2,0}=\|\widehat{V}^{k+1}\|_{2,0}
    ={\rm rank}(\overline{X}^{k+1})\leq\|U^{k+1}\|_{2,0}
    \le\|\overline{V}^k\|_{2,0}={\rm rank}(\overline{X}^{k}).
  \]
  Thus, $J_{\widehat{U}^{k+1}}\subseteq J_{\overline{V}^{k}}=J_{\overline{U}^{k}}$,
  and the second group of inclusions in \eqref{UVk-indxset} hold.
  Note that
  \[
  \|\overline{U}^{k+1}\|_{2,0}=\|\overline{V}^{k+1}\|_{2,0}
   ={\rm rank}(\widehat{X}^{k+1})\le
    \min(\|\widehat{U}^{k+1}\|_{2,0},\|V^{k+1}\|_{2,0}).
  \]
  So, $J_{\overline{U}^{k+1}}\subseteq J_{\widehat{U}^{k+1}}$.
  Since $J_{\widehat{U}^{k+1}}\subseteq J_{\overline{V}^{k}}=J_{\overline{U}^{k}}$,
  the first group of inclusions in \eqref{UVk-indxset} hold. Moreover,
  \begin{equation}\label{UV-F20value}
   \|\overline{U}^{k+1}\|_{2,0}\leq\|{V}^{k+1}\|_{2,0}\leq
   \|\widehat{V}^{k+1}\|_{2,0}=\|\widehat{U}^{k+1}\|_{2,0}
   \le\|U^{k+1}\|_{2,0}\leq\|\overline{U}^{k}\|_{2,0}.
  \end{equation}
  This means that the sequence $\{\|\overline{U}^{k}\|_{2,0}\}_{k\in\mathbb{N}}$
  is nonincreasing and convergent. By using \eqref{UV-F20value} again,
  $\lim_{k\to\infty}\|U^{k}\|_{2,0}=\lim_{k\to\infty}\|V^{k}\|_{2,0}
  =\lim_{k\to\infty}\|\overline{U}^k\|_{2,0}=\lim_{k\to\infty}\|\widehat{U}^k\|_{2,0}$.
  Since $\{\|\overline{U}^k\|_{2,0}\}$ is a nonnegative integer sequence,
  together with \eqref{UVk-indxset} we obtain the desired result.
 \end{proof}

  Proposition \ref{prop1-MAPM} (iii) states that the nonzero column indices of $\{(\overline{U}^k,\overline{V}^k)\}_{k\in\mathbb{N}}$ tend to be stable
  for all $k$ large enough. Inspired by this, we develop a hybrid AMM method
  in which, Algorithm \ref{MAPM} is first used to generate a point pair
  $(\overline{U}^k,\overline{V}^k)$ with a stable nonzero column index set,
  and then an alternating MM method similar to Algorithm \ref{AMM}
  with $(\overline{U}^k,\overline{V}^k)$ as a starting point is applied to
  \begin{equation}\label{Fmu-min}
    \min_{U\in\mathbb{R}^{n\times\kappa},V\in\mathbb{R}^{m\times\kappa}}F_{\mu}(U,V)
    \ \ {\rm with}\ \kappa=|J_{\overline{U}^k}|
  \end{equation}
  which is an unconstrained smooth problem. The iterates of the hybrid AMM method are as follows.
 \begin{algorithm}[H]
  \caption{\label{HMAP}{\bf (Hybrid AMM method for solving \eqref{MS-FL20})}}
  \textbf{Initialization:} Seek an output $(\overline{U}^k,\overline{V}^k)$
  with a stable $\kappa=J_{\overline{U}^k}=J_{\overline{V}^k}$ of Algorithm \ref{MAPM}
  for \eqref{MS-FL20}. Set $(U^{-1},V^{-1})=(U^0,V^0):=(\overline{U}^k,\overline{V}^k)$.
  Choose $\beta_0\in[0,\beta)$ with $\beta\in[0,1]$. Let $l:=0$. \\
 \textbf{while} the stopping conditions are not satisfied \textbf{do}
  \begin{itemize}
   \item[\bf 1.] Select $\gamma_{1,l}>\tau_{\!V^{l}}$.
                 Let $\widetilde{U}^l\!:=U^l+\beta_l(U^l-U^{l-1})$ and compute
                 \begin{equation}\label{Ul-subprob}
                   U^{l+1}\in\mathop{\arg\min}_{U\in\mathbb{R}^{n\times \kappa}}
                   \Big\{\langle\nabla_{\!1}F(\widetilde{U}^l,V^l),U\rangle
                   +\frac{\mu}{2}\|U\|_F^2+\frac{\gamma_{1,l}}{2}\|U\!-\!\widetilde{U}^l\|_F^2\Big\}.
                 \end{equation}

 \item[\bf 2.]  Select $\gamma_{2,l}>\tau_{\!U^{l+1}}$.
                Let $\widetilde{V}^l\!:=V^l+\beta_l(V^l-V^{l-1})$ and compute
                \begin{equation}\label{Vl-subprob}
                 V^{l+1}\!\in\mathop{\arg\min}_{\!V\in\mathbb{R}^{m\times \kappa}}
                 \Big\{\langle\nabla_{\!2}F(U^{l+1},\widetilde{V}^l),V\rangle
                 +\frac{\mu}{2}\|V\|_F^2 +\frac{\gamma_{2,l}}{2}\|V\!-\!\widetilde{V}^l\|_F^2\Big\}.
                \end{equation}

 \item[\bf 3.] Update $\beta_l$ by $\beta_{l+1}\in[0,\beta)$ and let $l\leftarrow l+1$.
 \end{itemize}
 \textbf{end while}
 \end{algorithm}
 \begin{remark}\label{remark-HMAP}
  {\bf(a)} When $r$ is a rough upper estimation for the true $r^*$,
  the value of $\kappa$ is usually much less than $r$ and is close to $r^*$
  due to the column $\ell_{2,0}$-norm term in \eqref{MS-FL20}. Thus,
  the computation cost of Algorithm \ref{HMAP} is expected to be much less than
  that of Algorithm \ref{AMM} and \ref{MAPM}.

  \noindent
  {\bf(b)} Since the subproblems \eqref{Ul-subprob} and \eqref{Vl-subprob}
  are strongly convex, by following the same arguments as those for Proposition \ref{prop1-UVk} and \ref{prop2-UVk},
  one may show that the sequence $\{(U^{l},V^{l})\}_{l\in\mathbb{N}}$ generated by
  Algorithm \ref{HMAP} is convergent, which along with Proposition \ref{prop1-Phi}
  means that its limit, say $(\overline{U},\overline{V})$, is also a critical point of $\Phi_{\lambda,\mu}$
  associated to $r=\kappa$. By \cite[Section 5.4]{Lee19}, the initial condition set
  where the sequence $\{(U^l,V^l)\}_{l\in\mathbb{N}}$ converges to a strict saddle point
  has a zero measure. Together with \cite[Theorem 3.1]{TaoPanBi19},
  when $f$ satisfies the assumption there, the limit $(\overline{U},\overline{V})$ with
  ${\rm rank}(\overline{U}\overline{V}^{\mathbb{T}})\le r^*$ will have a high probability
  to satisfy the error bound in \eqref{error-bound}.
 \end{remark}

 \section{Numerical experiments}\label{sec5}

 We shall test the performance of Algorithm \ref{AMM} and \ref{HMAP} by
 applying them to matrix completion problem in a general sampling scheme,
 and our codes can be downloaded from \url{https://github.com/SCUT-OptGroup/UVFL20}.
 Note that the matrix max-norm has been adopted as a convex surrogate
 for the rank function in \cite{Fang18,Lee10,Srebro10}, and the max-norm regularized approach
 was demonstrated in \cite{Fang18} to outperform the nuclear-norm convex relaxation method
 for matrix completion and collaborative filtering under non-uniform sampling schemes.
 To confirm the efficiency of the column $\ell_{2,0}$-norm regularized model \eqref{MS-FL20},
 we compare the numerical results with those of the ADMM developed in \cite{Fang18} for
 the SDP reformulation of the max-norm penalized LS model and those of the alternating
 least squares (ALS) method \cite{Hastie15} for the factorized model \eqref{MC-Fnorm}.
 The ALS method has the same iterate steps as Algorithm \ref{MAPM} does except that
 the column $\ell_{2,0}$-norm in $\widehat{\Phi}_{\lambda,\mu}$ and the proximal terms
 in Step 1 and 3 are removed. The numerical tests were all performed in MATLAB on a desktop computer
 running on 64-bit Windows Operating System with an Intel(R) Core(TM) i7-7700 CPU 3.60GHz
 and 16 GB RAM.

 \subsection{Matrix completion in a general sampling}\label{sec5.1}

  We assume that a random index set
  $\Omega=\big\{(i_t,j_t)\in[n]\times [m]\!: t=1,\ldots,p\big\}$
  is available, and that the samples of the indices are drawn independently
  from a general sampling distribution
  $\Pi=\{\pi_{kl}\}_{k\in[n],l\in[m]}$ on $[n]\times[m]$.
  We adopt the same non-uniform sampling scheme as in \cite{Fang18},
  i.e., for each $(k,l)\in[n]\times[m]$, take $\pi_{kl}=p_kp_l$ with
  \begin{equation}\label{sampling-scheme}
   \textrm{Scheme 1}\!:\
      p_k=\!\left\{\begin{array}{ll}
         2p_0& {\rm if}\ k\le\frac{n}{10} \\
         4p_0& {\rm if}\ \frac{n}{10}\le k\le \frac{n}{5}\\
          p_0& {\rm otherwise}\\
     \end{array}\right.\ {\rm or}\ \
   \textrm{Scheme 2}\!:\ p_k=\!
   \left\{\begin{array}{ll}
         3p_0& {\rm if}\ k\le\frac{n}{10} \\
         9p_0& {\rm if}\ \frac{n}{10}\le k\le \frac{n}{5}\\
          p_0& {\rm otherwise}\\
     \end{array}\right.
  \end{equation}
  where $p_0>0$ is a constant such that $\sum_{k=1}^{n}p_k=1$,
  and $p_l$ is defined in a similar way under the two schemes.
  For any $X\in\mathbb{R}^{n\times m}$,
  we denote by $X_{\Omega}\in\mathbb{R}^{n\times m}$ the projection
  of $X$ onto the set $\Omega$, i.e., $[X_{\Omega}]_{ij}=X_{ij}$ if $(i,j)\in\Omega$,
  otherwise $[X_{\Omega}]_{ij}=0$. Then, the function $f$ in \eqref{MC-Fnorm}
  and \eqref{MS-FL20} has the form
  \[
    f(X)=\frac{1}{2}\big\|X_{\Omega}-M_{\Omega}\big\|_F^2\quad\ {\rm for}\ X\in\mathbb{R}^{n\times m}
  \]
  where $M_{ij}$ for $(i,j)\in\Omega$ are the observed entries.
  For the simulated data, we assume that $M_{i_t,j_t}$ with $(i_t,j_t)\in\Omega$
  for $t=1,2,\ldots,p$ are generated via the following observation model
 \begin{equation}\label{observe}
    M_{i_t,j_t}=M_{i_t,j_t}^*+\sigma({\xi_{t}}/{\|\xi\|})\|M_{\Omega}^*\|_F,
 \end{equation}
 where $M^*\!\in\mathbb{R}^{n\times m}$ is the true matrix of rank $r^*$,
 $\xi=(\xi_1,\ldots,\xi_p)^{\mathbb{T}}$ is the noisy vector whose entries
 are i.i.d. random variables obeying $N(0,1)$, and $\sigma>0$ is the noise level.
 \subsection{Implementation of algorithms}\label{sec5.2}

  For the ADMM in \cite{Fang18}, we use the default stopping criterion,
  starting point and parameters.
 As mentioned before, the ADMM is developed for solving the SDP reformulation
  of the max-norm penalized LS model:
  \begin{equation}\label{SDP-maxnorm}
   \min_{Z\in\mathbb{S}^{n+m}}
   \Big\{\frac{1}{2}\big\|Z_{\Omega}^{12}-M_{\Omega}\big\|_F^2
   +\lambda\|{\rm diag}(Z)\|_\infty\ \ {\rm s.t.}\ \
   \|Z^{12}\|_{\infty}\le\alpha,\,Z\in\mathbb{S}_{+}^{n+m}\Big\}
  \end{equation}
  where $Z=\left(\begin{matrix}
            Z^{11}& Z^{12}\\
            (Z^{12})^{\mathbb{T}}& Z^{22}
         \end{matrix}\right)$ with $Z^{11}\in\mathbb{S}^{n}$, $Z^{22}\in\mathbb{S}^m$
  and $Z^{12}\in\mathbb{R}^{n\times m}$, and $\alpha>0$ is an upper bound
  for the elementwise $\ell_{\infty}$-norm of the true matrix $M^*$.
  It is worthwhile to point out that the code of ADMM is solving model \eqref{SDP-maxnorm}
  with a varying $\lambda$ instead of a fixed $\lambda$.

  Next we focus on the implementation details of other three algorithms.
  By comparing \eqref{optUk-equa}-\eqref{optVk-equa}  with the first-order optimality
  conditions of problem \eqref{MS-FL20}, it is not hard to obtain that
  \begin{subnumcases}{}
   E_{U}^{k+1}\in\nabla\!f(U^{k+1}(V^{k+1})^\mathbb{T})V^{k+1}+\mu U^{k+1}
        +\lambda\partial\|U^{k+1}\|_{2,0};\nonumber\\
   E_{V}^{k+1}\in[\nabla\!f(U^{k+1}(V^{k+1})^\mathbb{T})]^\mathbb{T}U^{k+1}+\mu V^{k+1}
     +\lambda\partial\|V^{k+1}\|_{2,0}\nonumber
   \end{subnumcases}
  where
  \begin{align*}
   E_{U}^{k+1}&:=\big[\nabla\!f(U^{k+1}(V^{k+1})^\mathbb{T})V^{k+1}
      -\nabla\!f(\widetilde{U}^{k}(V^{k})^\mathbb{T})V^{k}\big]+\gamma_{1,k}(\widetilde{U}^k\!-\!U^{k+1});\\
  E_{V}^{k+1}&:=\big[\nabla\!f(U^{k+1}(V^{k+1})^\mathbb{T})^\mathbb{T}U^{k+1}
  -\nabla\!f(U^{k+1}(\widetilde{V}^{k})^\mathbb{T})^\mathbb{T}U^{k+1}\big]
  +\gamma_{2,k}(\widetilde{V}^k\!-\!V^{k+1}).
  \end{align*}
  In view of this, we terminate Algorithm \ref{AMM} at $(U^{k},V^{k})$ when
  ${\rm rank}(X^{k})=\cdots={\rm rank}(X^{k-19})$ with $X^j=U^j(V^j)^{\mathbb{T}}$
  for $j=1,2,\ldots$ and either of the following conditions holds:
  \[
    \frac{\|(E_{U}^{k},E_{V}^{k})\|_F}{1+\|X^k\|_F}\le\epsilon_1\ \ {\rm or}\ \
    \frac{\max_{1\le i\le 9}|\Phi_{\lambda,\mu}(U^k,V^k)-\!\Phi_{\lambda,\mu}(U^{k-i},V^{k-i})|}
    {\max(1,\Phi_{\lambda,\mu}(U^k,V^k))}\le\epsilon.
  \]
 From the first-order optimality conditions of \eqref{Fmu-min}, we terminate
 Algorithm \ref{HMAP} at $(U^{l},V^{l})$ when
 \[
   \frac{\|(E_{U}^{l},E_{V}^{l})\|_F}{1+\|X^l\|_F}\le\epsilon_3\ \ {\rm or}\ \
   \frac{\max_{1\le i\le 9}|F_{\mu}(U^l,V^l)-F_{\mu}(U^{l-i},V^{l-i})|}{\max(1,F_{\mu}(U^l,V^l))}\le\epsilon.
 \]
 For the ALS method, we adopt a stopping criterion stronger than the one used in \cite{Hastie15}:
 \[
   {\rm rank}(X^{k})=\cdots={\rm rank}(X^{k-19})\ \ {\rm and}\ \ \frac{\|\overline{U}^{k}(\overline{V}^{k})^\mathbb{T}-\overline{U}^{k-1}(\overline{V}^{k-1})^\mathbb{T}\|_F^2}
  {\|\overline{U}^{k-1}(\overline{V}^{k-1})^\mathbb{T}\|_F^2}\le \epsilon_2.
 \]
 We always choose $\epsilon=10^{-4},\epsilon_1=10^{-3},\epsilon_3=5\times 10^{-3}$
 and $\epsilon_2=10^{-6}$ for the subsequent tests.

  For Algorithm \ref{AMM}, we set $\gamma_{1,k}=(1+\delta)L_{\!f}\|V^k\|^2$ and
  $\gamma_{2,k}=(1+\delta)L_{\!f}\|U^{k+1}\|^2$ with $\delta=10^{-6}$. For Algorithm \ref{HMAP},
  similar $\gamma_{1,l}$ and $\gamma_{2,l}$ are also used. We employ Nesterov's accelerated strategy
  \cite{Nesterov83} to yield $\beta_k$ of Algorithm \ref{AMM} and \ref{HMAP}, i.e.,
  $\beta_k=\frac{t_{k-1}-1}{t_{k}}$ with $t_{-1}=t_0=1$ and $t_{k+1}=\frac{1+\sqrt{4t_k^2+1}}{2}$.
  Though our convergence results require a restriction on $\beta_k$, numerical tests
  show that Algorithm \ref{AMM} and \ref{HMAP} still converge without it. In view of this,
  we do not impose any restriction on such $\beta_k$ for the subsequent tests,
  and leave this gap for a future research topic.
  The starting point $(U^0,V^0)$ of Algorithm \ref{AMM} is chosen to be
  $(P_{1}[\Sigma_{r}(M_{\Omega})]^{1/2},Q_{1}[\Sigma_{r}(M_{\Omega})]^{1/2})$
  and that of ALS is chosen to be $(P_{1},Q_{1})$, where $P_1$ and $Q_1$ are
  the matrix consisting of the first $r$ left and right singular vectors of $M_{\Omega}$, respectively.
  The starting point $(U^0,V^0)$ of Algorithm \ref{HMAP} is given by Algorithm \ref{MAPM}
  from the starting point $(P_{1},Q_{1})$ with $\underline{\gamma_{1}} =\underline{\gamma_{2}}=10^{-8}$,
  $\varrho=0.8$ and $\gamma_{1,0}=\gamma_{2,0}=0.01$.

  For the parameters of model \eqref{MS-FL20}, we always choose $r=\max(100,\min(\lceil0.5\min(n,m)\rceil,100))$
  and $\mu=10^{-8}$. Next we focus on the setting of $\lambda$. By Remark \ref{remark1-AMM} (b),
  $\overline{\lambda}:=0.5(1+\varsigma)(\mu\!+\!\gamma_{1,0})g^{\downarrow}_2$
  with $\varsigma=10^{-4}$ is a smaller $\lambda$ such that ${\rm rank}(U^1)\leq1$,
  and $\underline{\lambda}:=0.5(1-\varsigma)(\mu\!+\!\gamma_{1,0})g^{\downarrow}_r$ is a larger
  $\lambda$ such that ${\rm rank}(U^1)=r$, where $g=[\|G_1^0\|^2,\ldots,\|G_r^0\|^2]$ with $G^0$
  same as in Remark \ref{remark1-AMM} (b). This means that the desired $\lambda$ lies in the interval
  $[\underline{\lambda},\overline{\lambda}]$. Inspired by this, we take $\lambda_i=\overline{\lambda}+(1-i)\Delta\lambda$
  for $i=1,\ldots,n_{\lambda}$ with $\Delta\lambda=\frac{\overline{\lambda}-\underline{\lambda}}{n_{\lambda}-1}$,
  solve model \eqref{MS-FL20} associated to each $\lambda_i$ with Algorithm \ref{AMM},
  and pick the best from the obtained results associated to all $\lambda_i$ as the final output
  of Algorithm \ref{AMM}. Similarly, we take $\lambda_i=\overline{\lambda}+(1-i)\Delta\lambda$
  for $i=1,2,\ldots,n_{\lambda}$ with $\overline{\lambda}:=0.5(1+\varsigma)g^{\downarrow}_2$ and
  $\underline{\lambda}:=0.5(1-\varsigma)g^{\downarrow}_r$ to define $\Delta\lambda$,
  where the vector $g$ is determined by $G^0$ from Remark \ref{remark-MAPM} (b), solve model \eqref{MS-FL20}
  associated to each $\lambda_i$ with Algorithm \ref{MAPM}, and pick the best from the obtained results
  associated to all $\lambda_i$ as the final output of Algorithm \ref{MAPM}. For model \eqref{MC-Fnorm},
  since there is lack of such a good property, we choose the interval $[\underline{c},\overline{c}]$
  of $c_{\lambda}$ heuristically such that the rank of the solution to problem \eqref{MC-Fnorm}
  associated to $\lambda=c_{\lambda}{\rm SR}\|M_{\Omega}\|$ is included in $[0,r]$, where ${\rm SR}$ is
  the sample ratio. Then, we solve model \eqref{MC-Fnorm} with $\lambda_i=c_i{\rm SR}\|M_{\Omega}\|$
  for $i=1,\ldots,n_{\lambda}$ and pick the best from the obtained results associated to all $\lambda_i$
  as the final output of ALS, where $c_i=\overline{c}+(1\!-\!i)\Delta c$ with
  $\Delta c=\frac{\overline{c}-\underline{c}}{n_{\lambda}-1}$.
  As shown by the first two subfigures in Figure \ref{fig1},
  there is an interval of $\lambda$ such that Algorithm \ref{AMM} and \ref{HMAP} applied to
  \eqref{MS-FL20} with any $\lambda$ in this interval yield a lower relative error
  and a rank equal to $r^*$, while the last subfigure in Figure \ref{fig1} shows that
  there is an interval of $\lambda$ such that the outputs of ALS applied to model \eqref{MC-Fnorm}
  with one of $\lambda$ in the interval have a lower relative error but their ranks are higher than $r^*$.
 \begin{figure}[h]
  \setlength{\abovecaptionskip}{0.2pt}
  \centering
 \includegraphics[height=5.0cm,width=6.0in]{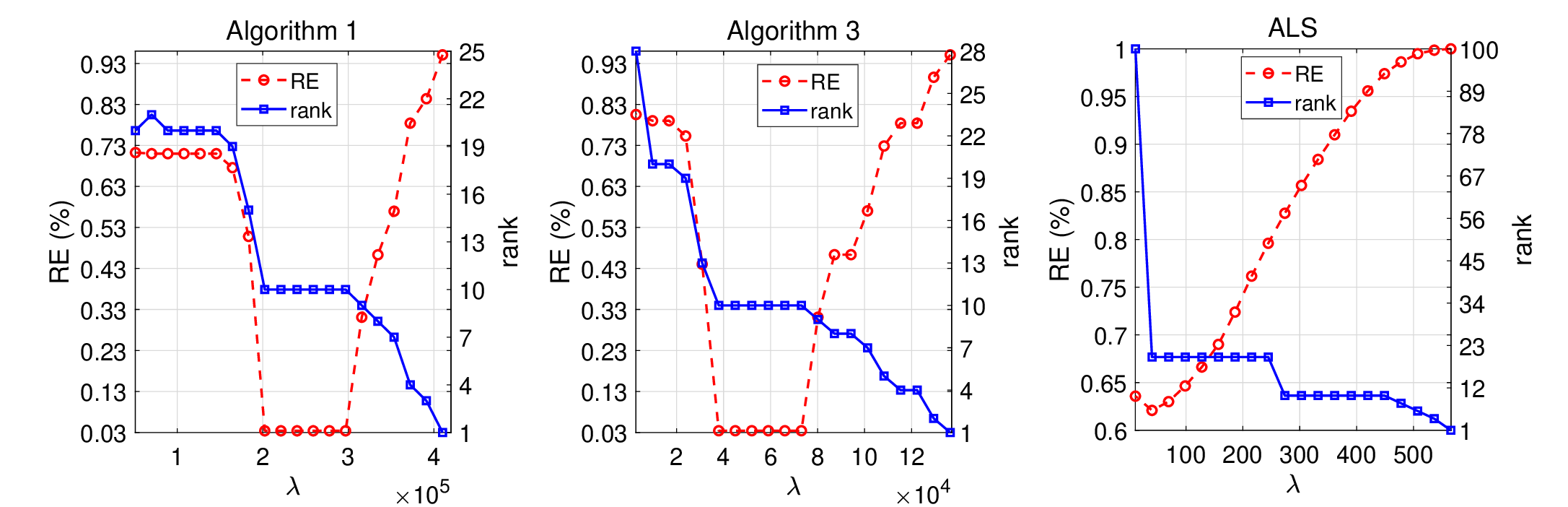}
  \caption{The relative error and rank curves of three solvers under different $\lambda$
  for ${\rm SR}=0.2$}
  \label{fig1}
 \end{figure}

  Next we present a rule to pick the best from the obtained results associated to all $\lambda_i$.
  A desirable solution of low-rank optimization problems is expected to have a low rank
  and a low relative error. Since the true $M^*$ is unknown, the relative error is unavailable.
  So, in the subsequent numerical experiments, we record the loss value and
  rank of the output of three algorithms associated to $\lambda_i$ in ${\rm loss}(i)$
  and ${\rm rank}(i)$ for $i=1,2,\ldots,\widehat{n}_{\lambda}$, where $\widehat{n}_{\lambda}$ is
  the number of different ranks. Here, when multiple outputs have the same rank, only the one with
  the smallest loss value is recorded.
  Then we set ${\rm loss}(0)=0.5\|M_{\Omega}\|_F^2$ and ${\rm rank}(0)=0$, and compute
  \[
    \vartheta(i)=\frac{|{\rm loss}(i\!-\!1)-{\rm loss}(i)|}{|{\rm rank}(i\!-\!1)-{\rm rank}(i)|}\ \ {\rm for}\ \ i=1,2,\ldots,\widehat{n}_{\lambda}.
  \]
  Clearly, $\vartheta(i)$ represents the relative change rate of the loss value with respect to
  the rank, obtained from model \eqref{MS-FL20} associated to $\lambda_i$.
  For the problem $\min_{U\in\mathbb{R}^{n\times r},V\in\mathbb{R}^{m\times r}}f(UV^{\mathbb{T}})$
  where $f$ is the least squares loss, it is easy to verify that its optimal value has the same
  magnitude as the noise does when $r\ge r^*$, but when $r<r^*$ it usually has a higher magnitude.
  This means that for the outputs of model \eqref{MS-FL20}, if their loss values have a larger change rate,
  it is highly possible for them to have a low rank; otherwise, they will have a high rank.
  Inspired by this, we find the smallest positive integer $i^*$ such that
  $\vartheta(i^*\!-\!1)/\vartheta(i^*)>5$ for simulated data, and the smallest
  positive integer $i^*$ such that $\vartheta(i^*\!-\!1)/\vartheta(i^*)>2$ for real data
  (by considering that the real data matrix usually has an approximately low rank),
  and then choose $\lambda_{i^*\!-\!1}$ as the best $\lambda$ for Algorithm \ref{AMM} and ALS,
  and the solution associated to $\lambda_{i^*\!-\!1}$ yielded by Algorithm \ref{MAPM} as
  the initial point of Algorithm \ref{HMAP}. The subsequent numerical tests
  always use $n_{\lambda}=21$ for three algorithms.

 \subsection{Numerical results for simulated data}\label{sec5.3}

  We test four solvers on simulated data under the non-uniform sampling
  scheme in \eqref{sampling-scheme}. We generate the true
  $M^*$ by $M^*=M_{L}^*(M_{R}^*)^{\mathbb{T}}$, where $M_{L}^*$ and $M_{R}^*$
  are an $n\times r^*$ matrix with each entry sampled independently
  from the standard normal distribution $N(0,1)$. Thus, $M^*\in\mathbb{R}^{n\times n}$
  is a rank $r^*$ matrix. The noisy observation entries $M_{i_t,j_t}$ with $(i_t,j_t)\in\Omega$
  are obtained from \eqref{observe} with $\sigma=0.1$ and $\xi_t\sim N(0,1)$,
  where the index set $\Omega$ is obtained in terms of Scheme 1.
  To evaluate the recovery results, we adopt the metric of relative error (RE)
  given by $\frac{\|X^{\rm out}-M^*\|_F}{\|M^*\|_F}$,
  where $X^{\rm out}$ represents the output of a solver. We consider different
  setting of $n, r^*$ and SR, and run simulation under each setting for five
  different instances.
 \setlength{\tabcolsep}{1mm}
 \begin{table}[h]
  \setlength{\belowcaptionskip}{-0.01cm}
  \centering
  \scriptsize
  \caption{\small Average RE and running time of four solvers for non-uniformly sampled synthetic data}\label{Simulated}
  \scalebox{1}{
 \begin{tabular}{cc|lcc|lcc|lccc|lcc}
 \Xhline{0.7pt}
 $n$&\!\! ($r^*$,\,SR)& \multicolumn{3}{l}{\qquad Algorithm \ref{AMM}}&\multicolumn{3}{l}{\qquad Algorithm \ref{HMAP}}
 &\multicolumn{4}{l}{\qquad ALS}&\multicolumn{3}{l}{\qquad ADMM}\\
 \cmidrule(lr){3-5} \cmidrule(lr){6-8} \cmidrule(lr){9-12} \cmidrule(lr){13-15}
  & &  RE&\!\! rank \!\! &time(s)&\!\!  RE \!\!\!\!& rank \!\!& time(s) &\ \ $c_{\lambda}$\!& RE& rank&\!\!\!\! time(s)&\!  RE&\!\!\! rank&\!\!\!\! time(s)\\
 \hline
 1000     &(8,0.10) &                 0.070  & 8 & 11.5 & {\color{red}\bf 0.066} & 8 & 5.73 &[0.2,10.5]  &0.847 & 17 & 30.0 &0.191&667&154.8\\
          &(8,0.15) & {\color{red}\bf 0.046} & 8 & 15.2 &                 0.047  & 8 & 6.79 &[0.1,\ 7.0] &0.845 & 14 & 30.6 &0.154&751&155.9\\
          &(8,0.20) & {\color{red}\bf 0.038} & 8 & 15.9 & {\color{red}\bf 0.038} & 8 & 6.86 &[0.1,\ 5.5] &0.871 & 13 & 35.0 &0.135&729&161.6\\
          &(8,0.25) & {\color{red}\bf 0.032} & 8 & 19.1 & {\color{red}\bf 0.032} & 8 & 8.44 &[0.1,\ 4.5] &0.916 & 10 & 40.4 &0.128&1000&164.2\\

          &(10,0.10) &                0.081  & 10 & 13.6 & {\color{red}\bf0.076} & 10 &5.74 &[0.3,10.5]&0.835 & 20 &34.1&0.195 &678 &159.5\\
          &(10,0.15) & {\color{red}\bf0.052} & 10 & 16.0 &                0.053  & 10 &6.79 &[0.2,7.0]& 0.872 & 16 &31.7  &0.160 &741 &157.0\\
          &(10,0.20) & {\color{red}\bf0.043} & 10 & 16.5 & {\color{red}\bf0.043} & 10 &7.21 &[0.1,5.5]& 0.923 & 12 &35.7  &0.142 &728 &155.9\\
          &(10,0.25) & {\color{red}\bf0.036} & 10 & 19.8 & {\color{red}\bf0.036} & 10 &8.65 &[0.1,4.5]& 0.817 & 16 &40.6  &0.132 &1000&157.4\\

          &(20,0.10) &  0.133                & 20 &19.6 & {\color{red}\bf 0.129}& 20 & 6.65 & [1.0,\ 11]& 0.880& 29 & 63.5 &0.253 &691&151.7\\
          &(20,0.15) &  0.084                & 20 &18.1 & {\color{red}\bf 0.082}& 20 & 7.57 & [0.4,7.5]& 0.795 & 33 & 76.1 &0.187 &765&151.8\\
          &(20,0.20) &{\color{red}\bf 0.065} & 20 &18.8 & {\color{red}\bf 0.065}& 20 & 7.78 & [0.2,5.5]& 0.820 & 20 & 39.5 &0.159 &719&156.9\\
          &(20,0.25) &{\color{red}\bf0.053}  & 20 &21.2 &                 0.054 & 20 & 9.18 & [0.1,4.5]& 0.883 & 24 & 49.9 &0.141 &1000&156.0\\

 \cmidrule(lr){1-15}
 3000     &(10,0.10) &                0.039  & 10 & 105.4 & {\color{red}\bf 0.038} & 10 & 55.8 & [0.2,\ 10]& 0.807 & 24 & 250.5&- &-&-\\
          &(10,0.15) &                0.029  & 10 & 145.1 & {\color{red}\bf 0.028} & 10 & 65.3 & [0.1,6.5]& 0.870 & 18 & 220.0 &- &-&-\\
          &(10,0.20) &{\color{red}\bf 0.024} & 10 & 159.9 & {\color{red}\bf 0.024} & 10 & 73.7 & [0.1,5.0]& 0.881 & 16 & 245.1 &- &-&-\\
          &(10,0.25) &{\color{red}\bf 0.020} & 10 & 260.5 & {\color{red}\bf 0.020} & 10 & 98.2 & [0.1,4.0]& 0.975 & 10 & 266.4 &- &-&-\\

          &(20,0.10) &                0.060  & 20 & 119.3 &  {\color{red}\bf0.055} & 20 & 58.0 & [0.3,\ 11]& 0.973 & 20& 203.3 &- &-&-\\
          &(20,0.15) & {\color{red}\bf0.041} & 20 & 147.2 &  {\color{red}\bf0.041} & 20 & 66.1 & [0.2,7.0]& 0.877 & 32 & 255.4&- &-&-\\
          &(20,0.20) & {\color{red}\bf0.034} & 20 & 172.9 &  {\color{red}\bf0.034} & 20 & 76.0 & [0.1,5.5]& 0.933 & 24 & 261.5 &- &-&-\\
          &(20,0.25) & {\color{red}\bf0.029} & 20 & 242.3 &  {\color{red}\bf0.029} & 20 & 98.9 & [0.1,4.0]& 0.833 & 32 & 304.2 &- &-&-\\

\cmidrule(lr){1-15}
 5000     &(10,0.10) &                 0.030  & 10 &318.4  & {\color{red}\bf 0.028} & 10 & 165.5 & [0.2,\ 11] & 0.850 & 20 & 703.2 &- &-&-\\
          &(10,0.15) & {\color{red}\bf 0.022} & 10 &443.4  & {\color{red}\bf 0.022} & 10 & 200.6 & [0.1,7.0]& 0.706 & 14 & 759.3 &- &-&-\\
          &(10,0.20) & {\color{red}\bf 0.018} & 10 &513.3  & {\color{red}\bf 0.018} & 10 & 234.7 & [0.1,5.5]& 0.983 & 10 & 837.0 &- &-&-\\
          &(10,0.25) & {\color{red}\bf 0.016} & 10 &913.4  & {\color{red}\bf 0.016} & 10 & 311.0 & [0.1,4.0]& 0.984 & 10 & 804.4 &- &-&-\\

          &(20,0.10) &                0.044  & 20 & 322.0 & {\color{red}\bf 0.041} &20& 178.0& [0.3,\ 11]&0.934 & 28 & 639.2 &- &-&-\\
          &(20,0.15) &{\color{red}\bf 0.031} & 20 & 455.7 & {\color{red}\bf 0.031} &20& 212.3& [0.1,7.0]& 0.918 & 28 & 943.0 &- &-&-\\
          &(20,0.20) &{\color{red}\bf 0.026} & 20 & 546.2 & {\color{red}\bf 0.026} &20& 242.9& [0.1,5.5]& 0.881 & 32 & 826.2 &- &-&-\\
          &(20,0.25) &{\color{red}\bf 0.022} & 20 & 856.5 & {\color{red}\bf 0.022} &20& 324.8& [0.1,4.0]& 0.975 & 20 & 903.7 &- &-&-\\
 \Xhline{0.7pt}
 \end{tabular}}
 \end{table}

   Table \ref{Simulated} reports the average RE, rank and running time (in seconds) of four solvers,
  where the results of ADMM are not reported for $n\ge 3000$ because it is too time-consuming.
  We see that for all test instances, the outputs of Algorithm \ref{AMM} and \ref{HMAP} not only
  have much lower RE than those of ALS and ADMM do, but also their ranks are equal to $r^*$,
  which coincides with their performance in Figure \ref{fig1} with $\lambda$ from the best interval.
  This means that the proposed column $\ell_{2,0}$-regularized factorization model is superior to
  another two models in capturing a low rank and low RE solution for non-uniformly sampled data.
  In Table \ref{Simulated}, the columns corresponding to ADMM show that the max-norm penalized model
  is suitable for non-uniform sampling in terms of RE, but can not promote a low-rank solution;
  while the columns corresponding to ALS show that model \eqref{MC-Fnorm} can promote low-rank solutions,
  but is not suitable for non-uniformly sampled data due to high RE.
  This coincides with the performance of the nuclear-norm and max-norm penalized models in \cite{Fang18}.

 In addition, for $r^*=5$ and $n=m=1000$, Figure \ref{fig2} plots the average RE
 over {\bf five} repetitions under ${\rm SR}=0.04,0.06,0.08,\ldots,0.2$. We see that
 under the two non-uniform sampling schemes, the relative errors yielded by
 four solvers decrease as the sampling ratio increases, but Algorithm \ref{AMM} and \ref{HMAP}
 have better performance than ADMM does, and the ALS method gives the worst results.
 \begin{figure}[h]
 \setlength{\abovecaptionskip}{0.2pt}
  \centering
 \includegraphics[height=6cm,width=6.0in]{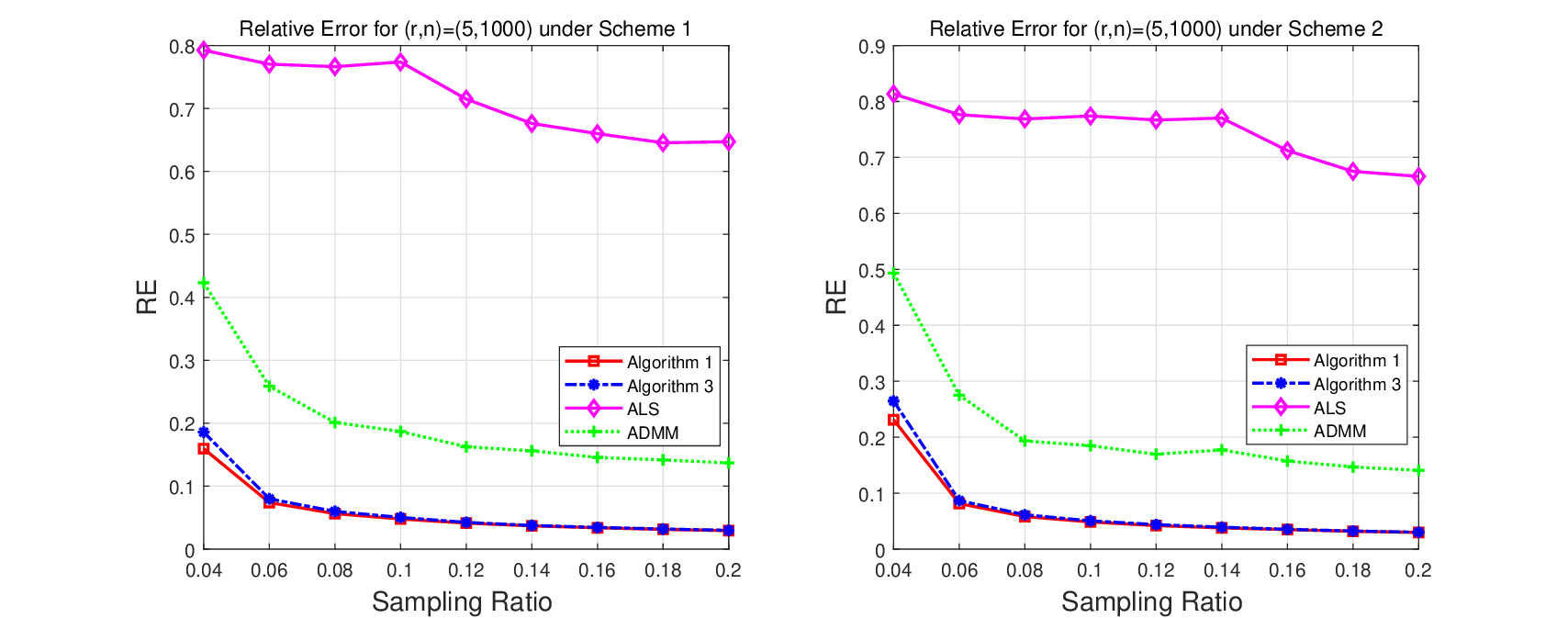}
  \caption{The relative errors of four solvers under different sampling ratios for noisy case}
  \label{fig2}
 \end{figure}
 \subsection{Numerical results for real data}\label{sec5.4}

  We test four methods with the matrix completion problem
  based on some real data sets, including the Jester joke dataset,
  the MovieLens dataset, and the Netflix dataset. For each data set,
  let $M^0$ be the original incomplete data matrix such that
  the $i$th row of $M^0$ corresponds to the ratings given by the $i$th user.
  We first consider the Jester joke dataset which is available through
  \url{http://www.ieor.berkeley.edu/~goldberg/jester-data/}. This dataset contains
  more than 4.1 million ratings for $100$ jokes from $73,421$ users. The whole
  Jester joke dataset contains three subdatasets: (1) jester-1: 24,983 users who rate 36 or more jokes;
  (2) jester-2: 23,500 users who rate 36 or more jokes; (3) jester-3: 24,938 users who rate
  between 15 and 35 jokes. More descriptions can be found in \cite{Chen12,Ma09,Toh10},
  where the nuclear-norm convex relaxation is used to study this dataset.
  Due to the large number of users, we first randomly select $n_u$ rows from $M^0$
  and then randomly permute the ratings from these users to generate
  $M\in\mathbb{R}^{{n_u}\times 100}$ as in \cite{Fang18}. Next, we adopt Scheme 1
  to generate a set $\Omega$ of observed indices.
  Since we can only observe the entry $(j,k)$ if $(j,k)\in\Omega$ and
  $M_{jk}$ is given, the actual sampling ratio is less than the input SR.

  Since the true $M^*$ is unknown for real datasets, we cannot compute
  the relative error as we did for simulated data. Similar to \cite{Toh10},
  we use the metric of the normalized mean absolute error
  \[
    {\rm NMAE}=\frac{\sum_{(i,j)\in\Gamma\backslash\Omega}|X^{\rm out}_{i,j}-M_{i,j}|}
   {|\Gamma\backslash\Omega|(r_{\rm max}-r_{\rm min})}
  \]
  to measure the accuracy of the output of an algorithm,
 where $\Gamma\!:=\{(i,j)\in[n_{u}]\times[100]\!: M_{ij}\ \textrm{is given}\}$
 denotes the set of indices for which $M_{ij}$ is given, and
 $r_{\rm min}$ and $r_{\rm max}$ denote the lower and upper bounds for the ratings, respectively.
 In the Jester joke dataset, the range is from $-10$ to $10$.
  \begin{table}[h]
  \setlength{\belowcaptionskip}{-0.01cm}
  \centering
  \scriptsize
  \caption{\small Average NMAE and running time of four solvers for the Jester joke dataset}\label{jester}
  \scalebox{1}{
 \begin{tabular}{cc|lcc|lcc|lccc|lcc}
 \Xhline{0.7pt}
 Dataset&\!\! ($n_u$,\,SR)& \multicolumn{3}{l}{\qquad Algorithm \ref{AMM}}&\multicolumn{3}{l}{\qquad Algorithm \ref{HMAP}}
 &\multicolumn{4}{l}{\qquad ALS}&\multicolumn{3}{l}{\qquad ADMM}\\
 \cmidrule(lr){3-5} \cmidrule(lr){6-8} \cmidrule(lr){9-12} \cmidrule(lr){13-15}
  & &\!\!\!\! NMAE&\!\!\!\! rank \!\!\!\! &time& \!\!\!\!  NMAE &\!\!\!\! rank \!\!\!\!& time & $\ \ \ \ c_{\lambda}$&\!\!\!\! NMAE&\!\!\!\! rank&\!\!\!\! time& NMAE& \!\!\!\! rank&\!\!\!\! time \\
 \hline 
 jester-1 &(1000,0.15) &      0.228            & 9 & 2.01 &                0.198  & 1 &0.29&[0.9,6.4] & 0.217 & 6 & 6.62 &{\color{red}\bf0.195}&100&25.3\\
          &(1000,0.20) &{\color{red}\bf 0.188} & 1 & 0.33 &{\color{red}\bf 0.188} & 1 &0.20&[0.8,4.8] & 0.221 & 1 & 1.46 &               0.190 &100&25.2\\
          &(1000,0.25) &{\color{red}\bf 0.187} & 1 & 0.21 &{\color{red}\bf 0.187} & 1 &0.16&[0.8,3.6] & 0.223 & 1 & 0.82 &               0.188 &100&25.2\\

          &(2000,0.15)  &{\color{red}\bf 0.195} & 1 & 0.82 &{\color{red}\bf 0.195}& 1 & 0.48 &[1.2,6.0] & 0.221 & 1 &1.79 &               0.196 &100&177.3\\
          &(2000,0.20)  &                0.194  & 1 & 0.58 &                0.194 & 1 & 0.43 &[1.0,4.8] & 0.222 & 1 &3.10 &{\color{red}\bf0.193}&100&176.4\\
          &(2000,0.25)  &{\color{red}\bf 0.189} & 1 & 0.50 &{\color{red}\bf 0.189}& 1 & 0.39 &[0.9,3.8] & 0.221 & 1 &3.16 &{\color{red}\bf0.189}&100&176.1\\

          &(4000,0.15)  &                0.203  & 8 & 7.46 &{\color{red}\bf 0.196} & 1 & 1.37 & [1.4,6.4]& 0.222 & 2 & 6.82 &- &-&-\\
          &(4000,0.20)  &{\color{red}\bf 0.190} & 1 & 1.22 &{\color{red}\bf 0.190} & 1 & 0.88 & [1.0,4.6]& 0.222 & 1 & 2.74 &- &-&-\\
          &(4000,0.25)  &{\color{red}\bf 0.185} & 1 & 1.41 &{\color{red}\bf 0.185} & 1 & 0.88 & [0.9,3.8]& 0.222 & 1 & 7.99 &- &-&-\\
 \cmidrule(lr){1-15}
 jester-2 &(1000,0.15)  &             0.197     & 1 & 0.29 &{\color{red}\bf 0.196}& 1 & 0.21 &[0.9,6.2]& 0.217 & 8 & 6.35 &{\color{red}\bf 0.196} &100&24.1\\
          &(1000,0.20)  &{\color{red}\bf 0.189} & 1 & 0.20 &{\color{red}\bf 0.189}& 1 & 0.14 &[0.8,4.6]& 0.223 & 1 & 0.82 &0.192   &100&24.2\\
          &(1000,0.25)  &{\color{red}\bf 0.187} & 1 & 0.20 &{\color{red}\bf 0.187}& 1 & 0.15 &[0.8,3.6]& 0.224 & 1 & 0.72 &0.190   &100&24.2\\

          &(2000,0.15)  &                0.196  & 1 & 0.71 &{\color{red}\bf 0.194} & 1 & 0.69 &[1.3,6.4]& 0.223 & 1 & 3.49&0.195 &100&178.9\\
          &(2000,0.20)  &{\color{red}\bf 0.189} & 1 & 1.02 &{\color{red}\bf 0.189} & 1 & 0.53 &[1.0,4.8]& 0.222 & 1 & 3.29 &0.192 &100&177.2\\
          &(2000,0.25)  &{\color{red}\bf 0.188} & 1 & 0.49 &{\color{red}\bf 0.188} & 1 & 0.37 &[0.8,3.8]& 0.219 & 4 & 7.42 &0.190 &100&177.6\\

          &(4000,0.15)  &{\color{red}\bf0.194}  & 1 & 1.22  &{\color{red}\bf 0.194} & 1 & 0.80 &[1.3,6.4]& 0.222 & 2 & 7.66 &- &-&-\\
          &(4000,0.20)  &{\color{red}\bf 0.187} & 1 & 0.86  &{\color{red}\bf 0.187} & 1 & 0.66 &[1.0,4.6]& 0.222 & 1 & 3.46 &- &-&-\\
          &(4000,0.25)  &{\color{red}\bf 0.186} & 1 & 0.84  &{\color{red}\bf 0.186} & 1 & 0.64 &[0.9,3.4]& 0.223 & 1 & 2.96 &- &-&-\\

\cmidrule(lr){1-15}
 jester-3 &(1000,0.15)  & 0.276 & 3 & 1.34 & 0.259 & 2 & 0.43 &[0.1,6.8]& 0.227 & 11& 6.08 &{\color{red}\bf 0.217} &88&23.8\\
          &(1000,0.20)  & 0.251 & 1 & 0.44 & 0.244 & 4 & 0.27 &[0.1,4.8]& 0.227 & 4 & 5.28 &{\color{red}\bf 0.212} &87&24.2\\
          &(1000,0.25)  & 0.301 & 5 & 2.33 & 0.263 & 2 & 0.25 &[0.1,4.4]& 0.227 & 8 & 7.18 &{\color{red}\bf 0.213} &91&24.2\\

          &(2000,0.15)  & 0.245 & 3 & 6.91 &  0.241 & 4 &1.30&[0.1,7.2]& 0.227 & 10 & 13.5 &{\color{red}\bf 0.217} &91&173.9\\
          &(2000,0.20)  & 0.258 & 6 & 9.85 &  0.248 & 5 &1.52&[0.1,5.0]& 0.223 & 8  & 17.0 &{\color{red}\bf 0.212} &91&176.1\\
          &(2000,0.25)  & 0.268 & 3 & 5.84 &  0.252 & 2 &1.16&[0.2,4.0]& 0.224 & 12 & 20.9 &{\color{red}\bf 0.213} &91&177.8\\

          &(4000,0.15)  & 0.249 & 1 & 4.97 &  0.245 & 2 & 1.65 & [0.3,7.0]& {\color{red}\bf0.229} & 5 & 21.6 &- &-&-\\
          &(4000,0.20)  & 0.258 & 2 & 6.53 &  0.255 & 2 & 2.04 & [0.3,5.2]& {\color{red}\bf0.231} & 3 & 11.3 &- &-&-\\
          &(4000,0.25)  & 0.234 & 2 & 4.64 &  0.241 & 2 & 1.89 & [0.3,4.0]& {\color{red}\bf0.226} & 8 & 21.7 &- &-&-\\
 \Xhline{0.7pt}
 \end{tabular}}
 \end{table}

 For the Jester joke dataset, we consider different settings of $n_u$ and SR,
 and report the average NMAE, rank and time (in seconds) for running the same setting
 {\bf five} times in Table \ref{jester}. Among others, the results of ADMM for $n_{u}=4000$
 are not reported since the adjusting scheme of $\lambda$ is not available in the code.
 We see that for jester-1 and jester-2,
 Algorithm \ref{AMM} and \ref{HMAP} yield comparable even a little better NMAE
 than ADMM does, but for jester-3 they give a little worse NMAE than
 ALS and ADMM do. For all settings, Algorithm \ref{AMM} and \ref{HMAP}
 yield much lower rank and require much less running time than ADMM does.
 The ALS method yields the worst NMAE for jester-1 and jester-2,
 and require comparable running time with that of Algorithm \ref{AMM} and \ref{HMAP}.

  Next we consider the MovieLens dataset from \url{http://www.grouplens.org/node/73}.
  The dataset contains two subdatasets: the Movie-100K dataset and the Movie-1M dataset,
  and the rating range is from $r_{\rm min}=1$ to $r_{\rm max}=5$. The Movie-100K dataset
  contains 100,000 ratings for 1682 movies by 943 users, while the latter contains
  1,000,209 ratings of 3900 movies made by 6040 users. For the Movie-100K dataset,
  we also consider the data matrix $\widetilde{M}^0=M^0-3$ so as to be consistent with
  the code of ADMM. We first randomly select $n_r$ users from $\widetilde{M}^0$
  and randomly select their $n_c$ column ratings, and then sample the observed entries
  with the schemes in \eqref{sampling-scheme}. Table \ref{Movie-100K} reports the averaged NMAE, rank and running time
  (in seconds) after running the setting $(n_r,n_c)=(943,1682)$ {\bf five} times.
  We see that Algorithm \ref{HMAP} yields a little better NMAE than other three solvers do,
  Algorithm \ref{AMM} gives worse NMAE than ADMM does for ${\rm SR}=0.1$ and $0.15$;
  and Algorithm \ref{AMM} and \ref{HMAP} yield the lowest rank solutions for all test problems,
  but ADMM gives the highest rank solutions.
 \begin{table}[h]
  \setlength{\belowcaptionskip}{-0.01cm}
  \centering
  \scriptsize
  \caption{\small Average NMAE and running time of four methods for Movie-100K dataset}\label{Movie-100K}
  \scalebox{1}{
 \begin{tabular}{cc|lcc|lcc|lccc|lcc}
 \Xhline{0.7pt}
 &SR& \multicolumn{3}{l}{\qquad Algorithm \ref{AMM}}&\multicolumn{3}{l}{\qquad Algorithm \ref{HMAP}}
 &\multicolumn{4}{l}{\qquad ALS}&\multicolumn{3}{l}{\qquad ADMM}\\
 \cmidrule(lr){3-5} \cmidrule(lr){6-8} \cmidrule(lr){9-12} \cmidrule(lr){13-15}
  & & NMAE& rank &time&  NMAE& rank& time  & \ \ \ $c_{\lambda}$& NMAE& rank& time&  NMAE& rank& time \\
 \hline
   Scheme 1&0.10 & 0.244 & 1 & 15.5 &{\color{red}\bf 0.231} & 1 & 2.93 & [0.4,9.2]& 0.248 & 10& 18.8 &0.232 &757&354.3\\
          &0.15  & 0.226 & 1 & 16.9 &{\color{red}\bf 0.219} & 1 & 3.08 & [0.4,5.4]& 0.247 & 1 & 3.63 &0.225 &867&353.9\\
          &0.20  & 0.216 & 1 & 16.3 &{\color{red}\bf 0.212} & 1 & 3.37 & [0.3,5.0]& 0.244 & 5 & 23.1 &0.220 &901&361.8\\
          &0.25  & 0.209 & 1 & 17.6 &{\color{red}\bf 0.207} & 1 & 3.59 & [0.3,2.6]& 0.242 & 1 & 5.56 &0.215 &927&361.8\\
  \Xhline{0.7pt}
  Scheme 2&0.10  & 0.246 & 1 & 15.3 & {\color{red}\bf 0.232} & 1 & 2.90 & [0.3,9.2]& 0.248 & 5  & 12.5 &0.233 &752&360.9\\
          &0.15  & 0.229 & 1 & 18.1 & {\color{red}\bf 0.221} & 1 & 3.05 & [0.3,6.2]& 0.247 & 12 & 19.2 &0.226 &851&361.6\\
          &0.20  & 0.217 & 1 & 17.0 & {\color{red}\bf 0.212} & 1 & 3.33 & [0.3,5.0]& 0.244 & 10 & 21.2 &0.221 &900&369.0\\
          &0.25  & 0.210 & 1 & 15.4 & {\color{red}\bf 0.208} & 1 & 3.39 & [0.3,2.6]& 0.244 & 1  & 4.93 &0.217 &922&366.7\\
 \Xhline{0.7pt}
 \end{tabular}}
 \end{table}

 \begin{table}[h]
  \setlength{\belowcaptionskip}{-0.01cm}
  \centering
  \scriptsize
  \caption{\small Average NMAE and running time of four methods for Movie-1M dataset \label{Movie-1M}}
  \scalebox{1}{
 \begin{tabular}{cc|lcc|lcc|lccc|lcc}
 \Xhline{0.7pt}
 ($n_r$,$n_c$)&\!\! SR & \multicolumn{3}{l}{\qquad Algorithm \ref{AMM}}&\multicolumn{3}{l}{\qquad Algorithm \ref{HMAP}}
 &\multicolumn{4}{l}{\ \ \  \qquad ALS}&\multicolumn{3}{l}{\qquad ADMM}\\
 \cmidrule(lr){3-5} \cmidrule(lr){6-8} \cmidrule(lr){9-12} \cmidrule(lr){13-15}
  & &\!\!\!\! NMAE&\!\!\! rank \!\!\!\! &time &  NMAE \!\!\!& rank \!\!\!& time& \ \ $c_{\lambda}$&\!\!\!\! NMAE& rank&\!\!\!\! time& NMAE&\!\!\! rank&\!\!\!\! time \\
 \Xhline{0.7pt}
 $1500\times 1500$ &0.10 & 0.242 & 1 & 22.8&{\color{red}\bf 0.229} & 1 & 4.41 & [0.5,8.2]& 0.251 & 1 & 6.50 &0.234 &850 &525.3\\
                   &0.15 & 0.226 & 1 & 25.4&{\color{red}\bf 0.218} & 1 & 4.44 & [0.4,5.2]& 0.250 & 1 & 6.55 &0.227 &999 &527.4\\
                   &0.20 & 0.212 & 1 & 24.4&{\color{red}\bf 0.209} & 1 & 4.49 & [0.3,3.4]& 0.249 & 1 & 6.07 &0.221 &1100&532.0\\
                   &0.25 & 0.207 & 1 & 29.7&{\color{red}\bf 0.205} & 1 & 4.85 & [0.3,3.0]& 0.247 & 1 & 7.65 &0.217 &1156&534.5\\
\Xhline{0.7pt}
 $2000\times 2000$ &0.10 & 0.228  & 1 & 41.5&{\color{red}\bf 0.219} & 1 & 8.11 &[0.8,9.4] & 0.251 &3 & 21.4 &0.231 &1245&1263.1 \\
                   &0.15 & 0.212  & 1 & 53.1&{\color{red}\bf 0.209} & 1 & 8.52 &[0.6,4.8] & 0.251 &1 & 10.9 &0.223 &1415&1271.9 \\
                   &0.20 & 0.207  & 1 & 44.1&{\color{red}\bf 0.204} & 1 & 8.70 &[0.5,3.8] & 0.250 &1 & 9.59 &0.219 &1524&1275.9 \\
                   &0.25 & 0.201  & 1 & 39.5&{\color{red}\bf 0.200} & 1 & 9.70 &[0.3,2.4] & 0.248 &1 & 12.8 &0.213 &1602&1363.3\\
\Xhline{0.7pt}
 $3000\times 3000$ &0.10 &                0.216 & 1 & 107.5&{\color{red}\bf 0.210 } & 1 & 23.9 &[1.2,7.8]& 0.253 &1  &22.8 & - &-&-\\
                   &0.15 &                0.204 & 1 & 104.8&{\color{red}\bf 0.202 } & 1 & 23.4 &[0.8,4.4]& 0.249 &1  &25.7 &- &-&-\\
                   &0.20 &                0.199 & 1 & 79.3 &{\color{red}\bf 0.197 } & 1 & 24.2 &[0.6,3.2]& 0.248 &1  &25.4 &- &-&-\\
                   &0.25 &{\color{red}\bf 0.195}& 1 & 81.4 &{\color{red}\bf 0.195 } & 1 & 28.0 &[0.4,2.2]& 0.242 &1  &29.1 &- &-&-\\
\Xhline{0.7pt}
$6040\times 3706$  &0.10 &                0.205 & 1 & 274.9& {\color{red}\bf 0.202 } & 1 & 56.9 &[1.3,6.8]& 0.251 & 1 & 50.1   & -&-&-\\
                   &0.15 &                0.197 & 1 & 218.2& {\color{red}\bf 0.196 } & 1 & 60.8 &[0.8,4.2]& 0.248 & 1 & 58.5  &- &-&-\\
                   &0.20 &{\color{red}\bf 0.194}& 1 & 174.5& {\color{red}\bf 0.194 } & 1 & 61.1 &[0.6,3.2]& 0.247 & 1 & 54.8    &- &-&-\\
                   &0.25 &{\color{red}\bf 0.192}& 1 & 177.1& {\color{red}\bf 0.192 } & 1 & 64.0 &[0.4,2.4]& 0.245 & 1 & 71.1 &- &-&-\\
 \Xhline{0.7pt}
 \end{tabular}}
 \end{table}

  For the Movie-1M dataset, we first randomly select $n_r$ users and their $n_c$
  column ratings from $M^0$, and then sample the observed entries with Scheme 1
  in \eqref{sampling-scheme}. We consider the setting of $n_r=n_c$ with $n_r=1500, 2000$
  or $3000$ and the setting of $(n_r,n_c)=(6040,3706)$. Table \ref{Movie-1M} reports
  the average NMAE, rank and running time (in seconds) after running {\bf five} times
  for each setting. We see that for this dataset, the solvers have similar
  performance as they do for the Movie-100K.

 \begin{table}[h]
 \setlength{\belowcaptionskip}{-0.01cm}
  \centering
  \scriptsize
  \caption{\small Average NMAE and running time of three methods for Netflix Dataset \label{Netflix}}
  \scalebox{1}{
 \begin{tabular}{ccc|lcc|lcc|lccc}
 \Xhline{0.7pt}
 &($n_r$,$n_c$)&\!\! SR & \multicolumn{3}{l}{\qquad Algorithm \ref{AMM}}&\multicolumn{3}{l}{\qquad Algorithm \ref{HMAP}}
 &\multicolumn{4}{l}{\qquad \ \ \ \  ALS}\\
 \cmidrule(lr){4-6} \cmidrule(lr){7-9} \cmidrule(lr){10-13}\\
  & &  & NMAE& rank  &time&   NMAE & rank & time  &\ \  $c_{\lambda}$& NMAE& rank& time \\
 \Xhline{0.7pt}
scheme 1
&$6000\times 6000 $
                    &0.10 & 0.228 & 1 & 440.1& {\color{red}\bf 0.218} & 1 & 121.4 &[1.1,8.4]& 0.246 &1& 110.1 \\
                &   &0.15 & 0.214 & 1 & 460.3& {\color{red}\bf 0.209} & 1 & 125.9 &[0.8,5.2]& 0.243 &1& 111.6 \\
                &   &0.20 & 0.208 & 1 & 496.6& {\color{red}\bf 0.204} & 1 & 127.4 &[0.6,3.4]& 0.244 &1& 113.9\\
                &   &0.25 & 0.203 & 1 & 452.1& {\color{red}\bf 0.201} & 1 & 133.2 &[0.5,3.0]& 0.242 &1& 127.9\\

&$8000\times 8000$  &0.10 & 0.214 & 1 & 726.4& {\color{red}\bf 0.208} & 1 & 140.7 &[1.3,9.0]& 0.246 &1 &106.5\\
                &   &0.15 & 0.206 & 1 & 762.0& {\color{red}\bf 0.203} & 1 & 148.5 &[0.8,5.4]& 0.244 &1 &104.4\\
                 &  &0.20 & 0.201 & 1 & 824.4& {\color{red}\bf 0.199} & 1 & 152.5 &[0.6,3.6]& 0.244 &1 &126.4 \\
                &   &0.25 & 0.198 & 1 & 760.4& {\color{red}\bf 0.196} & 1 & 163.4 &[0.5,2.8]& 0.241 &1 &130.6  \\

&$10000\times 10000$&0.10 & -  & - & - &{\color{red}\bf 0.207} & 1 & 217.3 &[1.4,8.2]& 0.245 & 1 &170.8\\
                &   &0.15 & -  & - & - &{\color{red}\bf 0.200} & 1 & 241.1 &[0.9,5.6]& 0.244 & 1 &158.1 \\
                &   &0.20 & -  & - & - &{\color{red}\bf 0.198} & 1 & 243.9 &[0.7,3.8]& 0.245 & 1 &176.2 \\
                &   &0.25 & -  & - & - &{\color{red}\bf 0.195} & 1 & 265.2 &[0.5,2.8]& 0.242 & 1 &191.4  \\
 \Xhline{0.7pt}
scheme 2 &$6000\times 6000$
                    &0.10 & 0.229 & 1 & 395.2&  {\color{red}\bf 0.219} & 1 & 78.3 &[1.0,8.4]& 0.246 &1 & 63.8 \\
                &   &0.15 & 0.216 & 1 & 433.1&  {\color{red}\bf 0.209} & 1 & 79.9 &[0.8,5.6]& 0.244 &1 & 63.4  \\
                &   &0.20 & 0.208 & 1 & 452.7&  {\color{red}\bf 0.204} & 1 & 83.5 &[0.6,3.6]& 0.245 &1 & 69.0\\
                &   &0.25 & 0.204 & 1 & 421.9&  {\color{red}\bf 0.201} & 1 & 84.0 &[0.4,3.0]& 0.243 &1 & 83.0\\

&$8000\times 8000$  &0.10 & 0.214  & 1 & 714.5  &{\color{red}\bf 0.209} & 1 & 141.2 &[1.2,9.0]& 0.246 &1 &102.8\\
                &   &0.15 & 0.206  & 1 & 712.8  &{\color{red}\bf 0.203} & 1 & 146.8 &[0.8,5.4]& 0.245 &1 &103.1 \\
                &   &0.20 & 0.201  & 1 & 803.2  &{\color{red}\bf 0.199} & 1 & 152.1 &[0.6,3.6]& 0.244 &1 &114.4 \\
                &   &0.25 & 0.198  & 2 & 763.6  &{\color{red}\bf 0.197} & 1 & 155.7 &[0.4,2.8]& 0.241 &1 &98.0  \\
  \Xhline{0.7pt}
 \end{tabular}}
 \end{table}

  We also consider the Netflix dataset in
  \url{https://www.kaggle.com/netflix-inc/netflix-prize-data\#qualifying.txt}.
  For this dataset, we first randomly select $n_r$ users and their $n_c$
  column ratings from $M^0$, and then sample the observed entries with
  the schemes in \eqref{sampling-scheme}. We consider the setting of $n_r=n_c$
  with $n_r=6000,8000$ and $10000$. Table \ref{Netflix} reports the average NMAE,
  rank and running time (in seconds) of three solvers after running {\bf five} times
  for each setting (the results of ADMM are not reported for these instances since
  it is too time-consuming). For this dataset, the three solvers have similar
  performance as they do for the MovieLens dataset. Among others, Algorithm \ref{HMAP}
  yields better outputs than other two solvers do, and it requires less half of the time
  than Algorithm \ref{AMM} does. So, Algorithm \ref{HMAP} has a remarkable
  advantage in running time for large-scale instances.

 From the numerical tests of the previous two subsections, we conclude that
 for simulated data, Algorithm \ref{AMM} and \ref{HMAP} are superior to ALS
 and ADMM in terms of rank and relative error; and for the three real datasets,
 Algorithm \ref{HMAP} is superior to other three solvers in terms of
 rank and NMAE except for jester-3, and its running time is also comparable with that of ALS.
 \section{Conclusion}\label{sec6}

  We have proposed a column $\ell_{2,0}$-norm regularized factorization model
  for low-rank matrix recovery to achieve the optimal (or true) rank from
  a rough upper estimation, so that the recent theoretical results for
  factorization models work fully in practice. We verify from theory
  that this model is superior to the squared Frobenius-norm regularized model \eqref{MC-Fnorm};
  for example, the critical points of model \eqref{MS-FL20} associated to a suitable $\lambda$
  and a tiny $\mu$ will have rank $r^*$ if their objective values are not greater than
  that of the projection of the noisy observation onto the rank $r^*$-constraint set,
  and under a suitable condition on $f$, the solution associated to a local minimizer
  of model \eqref{MS-FL20} with rank $r^*$ has a better error bound to the true $M^*$ than
  the solution associated to a local minimizer of model \eqref{MC-Fnorm} with rank $r^*$ does.
  We have developed an AMM method and a hybrid AMM method for computing this model,
  and provided their global convergence analysis. Numerical experiments are conducted
  on simulated data and real datasets for matrix completion problem with non-uniform sampling,
  and comparison results with the ALS \cite{Hastie15} and the ADMM \cite{Fang18}
  show that the proposed model has an advantage in promoting solutions with lower errors
  and ranks, and the hybrid AMM method is superior to other three solvers for
  most of test instances in terms of the error, rank and running time.
  The interesting future work is about the statistical study on the proposed model.

 \bigskip
 \noindent
 {\large\bf Acknowledgements}\ \ The authors would like to express their sincere thanks
 to Prof. Ethan X. Fang from Pennsylvania State University for providing them with the ADMM code
 for numerical comparison. The authors would like to express their sincere thanks to
 two anonymous referees and the Associated Editor for their helpful comments.

\bibliographystyle{siamplain}

 \bigskip
 \noindent
 {\bf\large Appendix A:}\\

 This part provides a lower bound to $M^*$ for the solution
 associated to a critical point of $F_{\lambda}$.
 \begin{aproposition}\label{critpoint-lowbound}
  Let $f(X)\!:=h(\mathcal{A}(X)-b)$ where $h\!:\mathbb{R}^p\to\mathbb{R}$ is an $L_{h}$-smooth function, $\mathcal{A}\!:\mathbb{R}^{n\times m}\to\mathbb{R}^p$ is the sampling operator,
  and $b=\mathcal{A}(M^*)+\omega$ for a noise vector $\omega\in\mathbb{R}^p$.
  Then, for any nonzero $(\overline{U},\overline{V})\in{\rm crit}F_{\lambda}$,
  it holds that
  \(
    \|\overline{U}\overline{V}^\mathbb{T}\!-M^*\|_F
    \geq\max\big(0,\frac{\lambda-\|\mathcal{A}^*\nabla h(\omega)\|}{L_{h}\|\mathcal{A}\|^2}\big).
  \)
 \end{aproposition}
 \begin{proof}
  Fix any nonzero $(\overline{U},\overline{V})\in{\rm crit}F_{\lambda}$.
  By the expression of $F_{\lambda}$, it is immediate to have that
  $\nabla\!f(\overline{U}\overline{V}^{\mathbb{T}})\overline{V}+\lambda\overline{U}=0$
  and $[\nabla\!f(\overline{U}\overline{V}^{\mathbb{T}})]^{\mathbb{T}}\overline{U}+\lambda\overline{V}=0$.
  Then, it holds that
  \[
    \langle \overline{U},\nabla\!f(\overline{U}\overline{V}^{\mathbb{T}})\overline{V}+\lambda\overline{U}\rangle=0
    \ \ {\rm and}\ \ \langle\overline{V},\big[\nabla\!f(\overline{U}\overline{V}^{\mathbb{T}})\big]^\mathbb{T}\overline{U}+\lambda\overline{V} \rangle=0.
  \]
  Recall that $\overline{U}^\mathbb{T}\overline{U}=\overline{V}^\mathbb{T}\overline{V}$ (see \cite[Proposition 4.3]{Li18}).
  Together with the last two equalities, we obtain
  \[
     -\lambda\|\overline{U}\|_F^2= \langle \overline{U},\nabla\!f(\overline{U}\overline{V}^{\mathbb{T}})\overline{V}\rangle
   \ge -\|\overline{U}\overline{V}^{\mathbb{T}}\|_*\|\nabla\!f(\overline{U}\overline{V}^{\mathbb{T}})\|
   \ge -\|\overline{U}\|_F^2\|\nabla\!f(\overline{U}\overline{V}^{\mathbb{T}})\|,
  \]
  where the last inequality is since
  \(
   \|X\|_*=\displaystyle{\min_{U\in\mathbb{R}^{n\times r},V\in\mathbb{R}^{m\times r}}}
   \big\{\frac{1}{2}(\|U\|_{F}^2+\|V\|_{F}^2)\ {\rm s.t.}\ X=UV^{\mathbb{T}}\big\}.
 \)
 Note that $\overline{U}\ne 0$ and $\overline{V}\ne 0$ since $(\overline{U},\overline{V})$ is nonzero
 and $\overline{U}^\mathbb{T}\overline{U}=\overline{V}^\mathbb{T}\overline{V}$.
 From the last inequality, it follows that $\|\nabla\!f(\overline{U}\overline{V}^{\mathbb{T}})\|\geq\lambda$.
 Substituting $\nabla\!f(X)=\mathcal{A}^*\nabla h(\mathcal{A}(X)-b)$ into this inequality
 and using the Lipschitz continuity of $\nabla h$ yields that
 \(
  \lambda\le L_{h}\|\mathcal{A}\|^2\|\overline{U}\overline{V}^\mathbb{T}\!-M^*\|_F+\|\mathcal{A}^*\nabla h(\omega)\|.
 \)
 This implies that the desired inequality holds. The proof is completed.
 \end{proof}

 Note that a suitably large $\lambda$ is necessary for model \eqref{MC-Fnorm} to achieve
 a solution with rank close to $r^*$ if the upper estimation $r$ is too rough. Together with
 the lower bound in Proposition \ref{critpoint-lowbound}, such $\lambda$ will lead to
 a large error bound to the true $M^*$ for the solution corresponding to the critical point
 of \eqref{MC-Fnorm}. Thus, it is very hard to achieve a solution with a small error
 and a rank close to the true $r^*$ by solving model \eqref{MC-Fnorm}.
 The last subfigure in Figure \ref{fig1} precisely shows this phenomenon.

 {
   \bigskip
  \noindent
  {\bf\large Appendix B: The proof of Proposition \ref{prop1-UVk}.}\\

  \begin{aproof}
  By the optimality of $U^{k+1}$ and the feasibility of $U^k$ to \eqref{Uk-subprob},
  it follows that
  \begin{align}\label{temp-ineq1}
   &\langle\nabla_{\!1}F(\widetilde{U}^k,V^k),U^{k+1}\rangle+\frac{\mu}{2}\|U^{k+1}\|_F^2
   +\frac{\gamma_{1,k}}{2}\|U^{k+1}\!-\!\widetilde{U}^k\|_F^2 +\lambda\|U^{k+1}\|_{2,0}\nonumber\\
   &\le\langle\nabla_{\!1}F(\widetilde{U}^k,V^k),U^{k}\rangle+\frac{\mu}{2}\|U^{k}\|_F^2
   +\frac{\gamma_{1,k}}{2}\|U^{k}\!-\!\widetilde{U}^k\|_F^2 +\lambda\|U^{k}\|_{2,0}.
  \end{align}
  By invoking inequality \eqref{FU} with $\gamma=\gamma_{1,k},V=V^k, U'=U^{k+1}$ and $U=\widetilde{U}^k$,
  we obtain
  \begin{align}\label{convex-monotone1}
   &F(U^{k+1},V^k)\le F(\widetilde{U}^{k},V^k)
    +\langle\nabla_{\!1}F(\widetilde{U}^{k},V^k),U^{k+1}\!-\!\widetilde{U}^k\rangle
    +\frac{\tau_{V^{k}}}{2}\|U^{k+1}\!-\!\widetilde{U}^k\|_F^2\nonumber\\
   &\le F(U^k,V^k)+\langle\nabla_{\!1}F(\widetilde{U}^{k},V^k),U^{k+1}-U^k\rangle
    +\frac{\tau_{V^{k}}}{2}\|U^{k}\!-\!\widetilde{U}^k\|_F^2+\frac{\tau_{V^{k}}}{2}\|U^{k+1}\!-\!\widetilde{U}^k\|_F^2
  \end{align}
  where the last inequality is by \eqref{nFU} with $\gamma=\tau_{V^{k}},V=\!V^{k},U=\!\widetilde{U}^{k}$ and
  $U'=U^k$. Along with \eqref{temp-ineq1},
 \begin{align}\label{FU-ineq1}
  F(U^{k+1},V^k)&\le F(U^k,V^k)+\frac{\mu}{2}\|U^{k}\|_F^2+\lambda\|U^{k}\|_{2,0}
                    +\frac{\gamma_{1,k}+\tau_{V^{k}}}{2}\|U^{k}\!-\!\widetilde{U}^k\|_F^2\nonumber \\
                &\quad-\frac{\gamma_{1,k}-\tau_{V^{k}}}{2}\|U^{k+1}\!-\!\widetilde{U}^k\|_F^2
                  -\frac{\mu}{2}\|U^{k+1}\|_F^2-\lambda\|U^{k+1}\|_{2,0}.
 \end{align}
 By using the same arguments as those for \eqref{FU-ineq1} and the optimality of $V^{k+1}$
  to \eqref{Vk-subprob}, it follows that
  \begin{align}\label{FV-ineq1}
   F(U^{k+1},V^{k+1})
   &\le F(U^{k+1},V^k)+\frac{\mu}{2}\|V^{k}\|_F^2+\lambda\|V^{k}\|_{2,0}
   -\frac{\mu}{2}\|V^{k+1}\|_F^2 -\lambda\|V^{k+1}\|_{2,0}\nonumber\\
   &\quad +\frac{\gamma_{2,k}+\tau_{U^{k+1}}}{2}\|V^{k}\!-\!\widetilde{V}^k\|_F^2
   -\frac{\gamma_{2,k}-\tau_{U^{k+1}}}{2}\|V^{k+1}\!-\!\widetilde{V}^k\|_F^2.
  \end{align}
   By substituting \eqref{FU-ineq1} into this inequality and using
   the definition of $\Phi_{\lambda,\mu}$, it follows that
  \begin{align}\label{key-ineq31}
   \Phi_{\lambda,\mu}(U^{k+1},V^{k+1})
   &\le\Phi_{\lambda,\mu}(U^{k},V^{k})+\frac{\gamma_{1,k}+\tau_{V^{k}}}{2}\|U^{k}\!-\!\widetilde{U}^k\|_F^2
     +\frac{\gamma_{2,k}+\tau_{U^{k+1}}}{2}\|V^{k}\!-\!\widetilde{V}^k\|_F^2\nonumber\\
   &\quad -\frac{\gamma_{1,k}-\tau_{V^{k}}}{2}\|U^{k+1}\!-\!\widetilde{U}^k\|_F^2
         -\frac{\gamma_{2,k}-\tau_{U^{k+1}}}{2}\|V^{k+1}\!-\!\widetilde{V}^k\|_F^2.
  \end{align}
  Together with $\widetilde{U}^k=U^k+\beta_k(U^k-U^{k-1})$
  and $\widetilde{V}^k=V^k+\beta_k(V^k\!-\!V^{k-1})$ and
  the definitions of $\alpha_{1,k}$ and $\alpha_{2,k}$,
  we deduce that for each integer $k\ge 0$ the left hand side of \eqref{descent-ineq1} is not more than
  \begin{align*}
  {\rm RHT}&=\frac{2\tau_{V^{k}}\beta_k^2-\rho_1\alpha_{1,k}}{2}\big\|U^{k}\!-\!U^{k-1}\big\|_F^2
     -\frac{(1-\rho_1)(\gamma_{1,k}-\tau_{V^{k}})}{2}\big\|U^{k+1}\!-\!U^k\big\|_F^2\\
   &\quad +(\gamma_{1,k}\!-\!\tau_{V^{k}})\beta_k\langle U^{k+1}\!-\!U^k,U^k\!-\!U^{k-1}\rangle
         +(\gamma_{2,k}\!-\!\tau_{U^{k+1}})\beta_k\langle V^{k+1}\!-\!V^k,V^k\!-\!V^{k-1}\rangle\\
   &\quad +\frac{2\tau_{U^{k+1}}\beta_k^2-\rho_2\alpha_{2,k}}{2}\|V^{k}\!-\!V^{k-1}\|_F^2
         -\frac{(1-\rho_2)(\gamma_{2,k}-\tau_{U^{k+1}})}{2}\|V^{k+1}\!-\!V^k\|_F^2\\
   &\le-\Big(\frac{\rho_1\alpha_{1,k}-2\tau_{V^{k}}\beta_k^2}{2}-\frac{\beta_k^2}{2t_1^k}\Big)
         \|U^{k}\!-\!U^{k-1}\|_F^2\nonumber\\
   &\quad -\frac{(1-\rho_1)(\gamma_{1,k}-\tau_{V^{k}})-t_1^k(\gamma_{1,k}\!-\!\tau_{V^{k}})^2}{2}
     \|U^{k+1}\!-\!U^k\|_F^2\\
   &\quad -\Big(\frac{\rho_2\alpha_{2,k}-2\tau_{U^{k+1}}\beta_k^2}{2}-\frac{\beta_k^2}{2t_2^k}\Big)
    \|V^{k}\!-\!V^{k-1}\|_F^2\nonumber\\
   &\quad -\frac{(1-\rho_2)(\gamma_{2,k}-\tau_{U^{k+1}})-t_2^k(\gamma_{2,k}\!-\!\tau_{U^{k+1}})^2}{2}
     \|V^{k+1}\!-\!V^k\|_F^2
  \end{align*}
  for any $t_1^k>0$ and $t_2^k>0$. In particular, taking
  $t_1^k=\frac{1-\rho_1}{\gamma_{1,k}-\tau_{V^{k}}}$ and
  $t_2^k=\frac{1-\rho_2}{\gamma_{2,k}-\tau_{U^{k+1}}}$ yields \eqref{descent-ineq1}.
  \end{aproof}

  \bigskip
  \noindent
  {\bf\large Appendix C: The proof of Proposition \ref{prop2-UVk}.}\\

  \begin{aproof}
  {\bf(i)} By following the same arguments as those for Proposition \ref{prop1-UVk},
  one may obtain
  \begin{align*}
   &\Xi_{\lambda,\mu}(U^{k+1},V^{k+1},U^k,V^k)-\Xi_{\lambda,\mu}(U^{k},V^{k},U^{k-1},V^{k-1})\nonumber\\
   &\le\frac{\rho_1\alpha_2}{2}\big(\|U^{k+1}\!-\!U^{k}\|_F^2-\|U^{k}\!-\!U^{k-1}\|_F^2\big)
       +\frac{\rho_2\alpha_2}{2}\big(\|V^{k+1}\!-\!V^{k}\|_F^2-\|V^{k}\!-\!V^{k-1}\|_F^2\big)\nonumber\\
   &\quad +\frac{\gamma_{1,k}+\tau_{V^{k}}}{2}\|U^{k}\!-\!\widetilde{U}^k\|_F^2
     +\frac{\gamma_{2,k}+\tau_{U^{k+1}}}{2}\|V^{k}\!-\!\widetilde{V}^k\|_F^2\nonumber\\
   &\quad -\frac{\gamma_{1,k}-\tau_{V^{k}}}{2}\|U^{k+1}\!-\!\widetilde{U}^k\|_F^2
         -\frac{\gamma_{2,k}-\tau_{U^{k+1}}}{2}\|V^{k+1}\!-\!\widetilde{V}^k\|_F^2.
  \end{align*}
  Then, using the same analysis technique as those for RHT after \eqref{key-ineq31}
  yields the result.

  \noindent
  {\bf(ii)-(iii)} Part (ii) holds by Proposition \ref{prop1-UVk} and
  the coerciveness of $\Xi_{\lambda,\mu}$. We next focus
  on the proof of part (iii). By part (i), the nonnegative sequence
  $\{\Xi_{\lambda,\mu}(U^{k},V^{k},U^{k-1},V^{k-1})\}_{k\in\mathbb{N}}$
  is nonincreasing. So, the limit $\varpi^*$ exists.
  Fix an arbitrary $(\overline{U},\overline{V},\overline{Y},\overline{Z})\in\Upsilon$.
  There is an index set $\mathcal{K}\subseteq\mathbb{N}$ such that
  $({U}^{k},{V}^{k},{U}^{k-1},{V}^{k-1})
 \rightarrow(\overline{U},\overline{V},\overline{Y},\overline{Z})$ when
 $\mathcal{K}\ni k\rightarrow\infty$.
 By the feasibility of $\overline{U}$ to \eqref{Uk-subprob}, for each $k$,
 \begin{align*}
  &\langle\nabla_{\!1}F(\widetilde{U}^{k-1},V^{k-1}),U^{k}\rangle
  +\frac{\mu}{2}\|U^{k}\|_F^2+\lambda\|{U}^{k}\|_{2,0}+
  \frac{\gamma_{1,k-1}}{2}\|U^{k}-\widetilde{U}^{k-1}\|_F^2\\
  &\le\langle\nabla_{\!1}F(\widetilde{U}^{k-1},V^{k-1}),\overline{U}\rangle
    +\frac{\mu}{2}\|\overline{U}\|_F^2+\lambda\|\overline{U}\|_{2,0}
   +\frac{\gamma_{1,k-1}}{2}\|\overline{U}-\widetilde{U}^{k-1}\|_F^2.
 \end{align*}
 Passing to the limit $k\xrightarrow[\mathcal{K}]{}\infty$
 and using the boundedness of $\gamma_{1,k}$,
 \(
  \limsup_{k\xrightarrow[\mathcal{K}]{}\infty}\|{U}^{k}\|_{2,0} \leq\|\overline{U}\|_{2,0}.
 \)
 In addition, by the lower semicontinuity of $\|\cdot\|_{2,0}$, we have
 $\liminf_{k\xrightarrow[\mathcal{K}]{}\infty}\|{U}^{k}\|_{2,0}\ge\|\overline{U}\|_{2,0}$.
 Thus, $\lim_{k\xrightarrow[\mathcal{K}]{}\infty}\|{U}^{k}\|_{2,0} =\|\overline{U}\|_{2,0}$.
 Similarly, we also have $\lim_{k\xrightarrow[\mathcal{K}]{}\infty}\|V^{k}\|_{2,0} =\|\overline{V}\|_{2,0}$.
 Together with the expression of $\Xi_{\lambda,\mu}$,
 $\lim_{k\xrightarrow[\mathcal{K}]{}\infty}\Xi_{\lambda,\mu}({U}^{k},{V}^{k},{U}^{k-1},{V}^{k-1})
 =\Xi_{\lambda,\mu}(\overline{U},\overline{V},\overline{Y},\overline{Z})$.
 Since the limit of the sequence
 $\{\Xi_{\lambda,\mu}({U}^{k},{V}^{k},{U}^{k-1},{V}^{k-1})\}_{k\in\mathbb{N}}$
 is exactly $\varpi^*$. This implies that $\Xi_{\lambda,\mu}(\overline{U},\overline{V},\overline{Y},\overline{Z})=\varpi^*$.
 By the arbitrariness of $(\overline{U},\overline{V},\overline{Y},\overline{Z})$
 on the set $\Upsilon$, it follows that $\Xi_{\lambda,\mu}\equiv\varpi^*$ on $\Upsilon$.

 \noindent
 {\bf(iv)} By the expression of $\Xi_{\lambda,\mu}$ and \cite[Exercise 8.8]{RW98},
 for any $(U,V,U',V')$ it holds that
 \begin{align}\label{gradXi-UV1}
 \partial\Xi_{\lambda,\mu}(U,V,U',V')
 =\left[\begin{matrix}
   \nabla f(UV^\mathbb{T})V+\!\mu U + \lambda \partial \|U\|_{2,0}+\rho_1\alpha_2(U-U')\\
   (\nabla f(UV^\mathbb{T}))^{\mathbb{T}}U+\!\mu V+ \lambda\partial\|V\|_{2,0}+\rho_2\alpha_2(V-V')\\
   \rho_1\alpha_2(U'-U)\\
   \rho_2\alpha_2(V'-V)
   \end{matrix}\right].
 \end{align}
  In addition, from the definition of $U^{k+1}$ and $V^{k+1}$ in Step 1 and 2,
  for each $k\in\mathbb{N}$ it follows that
  \begin{subnumcases}{}
   \label{optUk-equa}
   0\in\nabla\!f(\widetilde{U}^{k}(V^{k})^\mathbb{T})V^{k}+\mu U^{k+1}
    +\gamma_{1,k}(U^{k+1}-\widetilde{U}^k)+\lambda\partial\|U^{k+1}\|_{2,0};\quad\\
   \label{optVk-equa}
  0\in[\nabla\!f(U^{k+1}(\widetilde{V}^{k})^\mathbb{T})]^\mathbb{T}U^{k+1}+\mu V^{k+1}
    +\gamma_{2,k}(V^{k+1}-\widetilde{V}^k)+\lambda\partial\|V^{k+1}\|_{2,0}.
  \end{subnumcases}
  Hence,
  \(
    \big(\Gamma_U^{k+1},\Gamma_V^{k+1},\rho_1\alpha_2(U^{k}\!-\!U^{k+1}),
    \rho_2\alpha_2(V^{k}\!-\!V^{k+1})\big)
    \in\partial\Xi_{\lambda,\mu}(U^{k+1},V^{k+1},U^{k},V^{k})
  \)
 with
 \begin{subnumcases}{}
   \Gamma_U^{k+1}\!=\!\nabla f(U^{k+1}({V}^{k+1})^\mathbb{T})V^{k+1}\!-\!\nabla f(\widetilde{U}^{k}(V^{k})^\mathbb{T})V^{k}\!-\!\gamma_{1,k}(U^{k+1}\!-\!\widetilde{U}^k)
         \!+\! \rho_1\alpha_2(U^{k+1}\!-\!U^{k});\nonumber\\
   \Gamma_V^{k+1}\!=\!\big[\!\nabla\! f(U^{k+1}({V}^{k+1})^\mathbb{T}\!)\!-\!\nabla f(U^{k+1}(\!\widetilde{V}^{k}\!)^\mathbb{T}\!)\big]^\mathbb{T}U^{k+1}
         \!-\!\gamma_{2,k}(V^{k+1}\!-\!\widetilde{V}^k)\!+\!\rho_2\alpha_2(V^{k+1}\!-\!V^{k}).\nonumber
 \end{subnumcases}
  This means that the distance ${\rm dist}\big(0,\partial \Xi_{\lambda,\mu}(U^{k+1},V^{k+1},U^k,V^k)\big)$
  is upper bounded by
  \begin{align*}
   &\sqrt{\|\Gamma_U^{k+1}\|_F^2+\|\Gamma_V^{k+1}\|_F^2+\rho_1^2\alpha_2^2\|U^{k}-U^{k+1}\|_F^2
      +\rho_2^2\alpha_2^2\|V^{k}-V^{k+1}\|_F^2}\\
   &\leq (\tau_{V^{k}}\!+\!\gamma_{1,k})\|U^{k+1}\!-\widetilde{U}^{k}\|_F +2\rho_1\alpha_2\|U^{k+1}\!-\!U^{k}\|_F
        \!+(\tau_{U^{k+1}}\!+\!\gamma_{2,k})\|V^{k+1}\!-\widetilde{V}^{k}\|_F \\
   &\quad+(c_f+2\rho_2\alpha_2\!+\!\sqrt{\tau_{U^{k+1}}\tau_{V^{k}}})\|V^{k+1}-V^{k}\|_F\\
   &\le(\tau_{V^{k}}+\gamma_{1,k}+2\rho_1\alpha_2)\|U^{k+1}\!-\!{U}^{k}\|_F +(\tau_{V^{k}}+\gamma_{1,k})\beta_k\|U^{k}-U^{k-1}\|_F\!\\
   &\quad +(c_f\!\!+\!2\rho_2\alpha_2+\tau_{U^{k+1}}+\gamma_{2,k}\!+\!\sqrt{\tau_{U^{k+1}}\tau_{V^{k}}})\|V^{k+1}\!-\!V^{k}\|_F
     +(\tau_{U^{k+1}}\!+\!\gamma_{2,k})\beta_k\|V^{k}\!-\!{V}^{k-1}\|_F.
  \end{align*}
  This implies that the desired inequality holds. Thus, we complete the proof.
 \end{aproof}
 }
 \end{document}